\documentclass[12pt,oneside,english]{amsart}
\usepackage{amssymb,amsmath,latexsym,amsthm}

\usepackage[a4paper,top=2cm,bottom=2cm,left=3cm,right=3cm,marginparwidth=1.75cm]{geometry}
\usepackage{graphicx}

\usepackage{xcolor}

\newtheorem{theorem}{Theorem}[section]
\newtheorem{corollary}[theorem]{Corollary}
\newtheorem{lemma}[theorem]{Lemma}
\newtheorem{prop}[theorem]{Proposition}
\newtheorem{quest}[theorem]{Question}
\theoremstyle{definition}
\newtheorem{definition}[theorem]{Definition}

\newcommand{\PSH}{{\rm PSH}}
\newcommand{\vol}{{\rm Vol}}
\newcommand{\capa}{{\rm Cap}}
\newcommand{\Ric}{{\rm Ric}}
\newcommand{\Ent}{{\rm Ent}}
\newcommand{\Amp}{{\rm Amp}}

\newcommand{\bR}{\mathbb{R}}
\newcommand{\bN}{\mathbb{N}}
\newcommand{\bC}{\mathbb{C}}
\newcommand{\cH}{\mathcal{H}}
\newcommand{\cE}{\mathcal{E}}
\newcommand{\cM}{\mathcal{M}}
\newcommand{\f}{\varphi}
\newcommand{\weak}{\rightharpoonup}

\numberwithin{equation}{section}

 \usepackage{hyperref}
\hypersetup{  
    pdftitle={Geodesic distance},    
    pdfauthor={Di Nezza and Lu},     
    colorlinks=true,       
   linkcolor=black,          
    citecolor=black,        
    filecolor=black,      
    urlcolor=black}

\subjclass[2010]{32W20, 32U05, 32Q15, 35A23}

\keywords{K\"ahler manifold, geodesic, envelope, Darvas metric, contact set, entropy}

\begin{document}

 \title[Geodesic distance and Monge-Amp\`ere measures on contact sets]{Geodesic distance and Monge-Amp\`ere measures on contact sets}
\author{Eleonora Di Nezza \& Chinh H. Lu}

\thanks{The authors are partially supported by the ANR project PARAPLUI ANR-20-CE40-0019}

\address{Ecole Polytechnique and Sorbonne Universit\'e}

\email{eleonora.di-nezza@polytechnique.edu}
\email{eleonora.dinezza@imj-prg.fr}

\address{Universit\'e Paris-Saclay, CNRS, Laboratoire de Math\'ematiques d'Orsay, 91405, Orsay, France.}

\email{hoang-chinh.lu@universite-paris-saclay.fr}

\date{\today}

\dedicatory{Dedicated to L\'aszl\'o Lempert on the occasion of his 70 birthday.}

\maketitle

\begin{abstract}
We prove a geodesic distance formula for quasi-psh functions with finite entropy, extending results by Chen and Darvas. We work with big and nef cohomology classes: a key result we establish is the convexity of the $K$-energy in this general setting. We then study Monge-Amp\`ere measures on contact sets, generalizing a recent result by the first author and Trapani. 
\end{abstract}

\tableofcontents

\section{Introduction}
Let $X$ be a compact K\"ahler manifold of dimension $n$ and fix a K\"ahler metric $\omega$. By the $dd^c$ lemma  there is a  correspondence between the space of K\"ahler metrics in $\{\omega\}$ and the space of K\"ahler potentials 
\[
\mathcal{H} := \{u\in C^{\infty}(X,\mathbb{R}) \; : \; \omega_u >0 \}.  
\]Motivated by the study of canonical metrics on $X$, in \cite{Mab87} Mabuchi introduced a Riemannian structure on the space of K\"ahler potentials, giving rise to the notion of geodesics connecting two  elements in $\mathcal{H}$. As discovered later by Semmes \cite{Sem92} and Donaldson \cite{Don99}, the geodesic equation can be formulated as a degenerate homogeneous complex Monge-Amp\`ere equation. 

By now the optimal $C^{1,1}$ regularity of geodesic segments is well-known  (see \cite{Chen00}, \cite{Blo12}, \cite{LV13}, \cite{DL12}, \cite{CTW18}).   Darvas  has introduced in \cite{Dar15} a family of $L^p$ type Finsler metrics on $\mathcal{H}$, generalizing the $L^2$ metric of Mabuchi, and studied their completions. These metrics have found numerous spectacular applications in K\"ahler geometry (see \cite{BBJ15}, \cite{DR17}, \cite{BDL17,BDL20}).    As shown by Chen \cite{Chen00} and  Darvas \cite{Dar15} the $L^p$ distance between two potentials $u_0,u_1$ with boundded Laplacian is realized by the geodesic segment $u_t$:
\begin{equation}
	\label{eq: geodesic distance intro}
	d_p(u_0,u_1)^p= \int_X |\dot{u}_t|^p (\omega+dd^c u_t)^n, \; \forall t \in [0,1]. 
\end{equation}
In \cite{DNL20}, we have extended the study of these metrics for a big and nef cohomology class $\{\theta\}$. We proved in that paper that \eqref{eq: geodesic distance intro} holds for $u_j =P_{\theta}(f_j)$, $j=0,1$, where $f_j$ are smooth functions on $X$.  

The main goal of this paper is to investigate the following 
\begin{quest}
	Under what condition on $u_0,u_1\in \mathcal{E}^1(X,\theta)$, does \eqref{eq: geodesic distance intro} hold? 
\end{quest}
Berndtsson has shown in \cite{Bern18} that \eqref{eq: geodesic distance intro} holds for geodesics with continuous time derivatives. 
By convexity of $t\mapsto u_t$ the left and right derivatives $\dot{u}_t^{\pm}(x)$ exist provided that $u_0(x)$ and $u_1(x)$ are finite. 
An example in \cite{Lem21} shows that there is a bounded psh geodesic such that the two directional derivatives differ $(\omega+dd^c u_t)^n$-almost everywhere for any $t$. Another similar example was given in \cite{Dar15} showing that \eqref{eq: geodesic distance intro} does not hold even for a bounded geodesic. In these examples the Monge-Amp\`ere measures $\omega_{u_0}^n$ and $\omega_{u_1}^n$ do not have density with respect to Lebesgue measure.  
\cite[Lemma 10.2]{Lem21} seems to suggest that \eqref{eq: geodesic distance intro} should hold when the measures $(\omega+ dd^c u_t)^n$ are uniformly absolutely continuous with respect to Lebesgue measure.  

In the general context of big and nef cohomology classes, our first main result confirms the above expectation: 
\begin{theorem}\label{thm: geodesic distance intro}
Assume $\theta$ is a smooth closed real $(1,1)$-form representing a big and nef class. Fix $p\geq 1$,  $u_0,u_1\in \Ent(X,\theta)$ and let $u_t$ be the psh geodesic connecting $u_0$ to $u_1$. If $u_0-u_1$ is bounded then 
\begin{equation*}
 \int_X |\dot{u}_t|^p (\theta +dd^c u_t)^n\; \text{is constant in}\; t \in [0,1].
\end{equation*}
If in addition $u_0,u_1\in \cE^p(X,\theta)$, then 
$$d_p^p(u_0,u_1)= \int_X |\dot{u}_t|^p (\theta +dd^c u_t)^n,\quad \forall t\in [0,1].$$\end{theorem} 
Here $\Ent(X,\theta)$ consists of functions $u\in \cE(X,\theta)$ whose Monge-Amp\`ere measure has finite entropy: 
\[
\Ent(\omega^n, \theta_u^n):= \int_X \log \left (\frac{\theta_u^n}{\omega^n} \right) \theta_u^n <+\infty.  
\] 
For more details on finite entropy potentials and finite energy classes $\cE^p(X,\theta)$, we refer to Section \ref{sec:prelim}. It is likely that Theorem \ref{thm: geodesic distance intro} holds for convex weights $\chi$ considered in \cite{Dar15}.

 In the K\"ahler case (i.e. when $\theta=\omega$), a crucial step in  our proof of Theorem \ref{thm: geodesic distance intro} is to prove that  the left and right derivatives $\dot{u}_t^{\pm}$ coincide almost everywhere with respect to $\omega_{u_t}^n$. This is done by using \cite[Lemma 4.11]{BDL17} and  the convexity of the extended Mabuchi K-energy (see \cite{BB17}, \cite{CLP16}, \cite{BDL17}).  
 In order to run the same arguments for  a big and nef cohomology class,  we introduce a natural extension of the K-energy and prove that it is convex along plurisubharmonic geodesics (Theorem \ref{thm: convexity of Mabuchi nef}), paving the way to studying constant scalar curvature metrics on singular K\"ahler varieties.


 We next discuss an application of Theorem \ref{thm: geodesic distance intro}. Given a real function $f$ of class $C^2$ on $X$,  one can show, by a standard balayage argument,  that the Monge-Amp\`ere measure of the envelope $P_{\omega}(f)$ is supported on the contact set $\mathcal{C} =\{P_{\omega}(f) =f\}$.  By proving that the Laplacian of  $P_{\omega}(f)$ is bounded,  Berman has shown in \cite{Ber19} that
 \[
 (\omega+dd^c P_{\omega}(f))^n = {\bf 1}_{\mathcal{C}} (\omega+dd^c f)_+^n,
 \] 
 where for a smooth real $(1,1)$-form $\alpha$ we let $\alpha_+= \alpha$ if $\alpha \geq 0$ and $\alpha_+=0$ otherwise. By a polarization argument  it is shown in \cite{DNT19} that the same conclusion holds for a big cohomology class $\{\theta\}$.  A crucial step in the proof of \cite{DNT19} is to show that if $u,v\in \PSH(X,\omega)$, $v$ has bounded Lplacian, and $u\leq v$, then we have the equality of Monge-Amp\`ere measures on the contact set: 
 \begin{equation}
 	\label{eq: MA contact}
 	{\bf 1}_{\{u=v\}} (\omega+dd^c u)^n = {\bf 1}_{\{u=v\}} (\omega+dd^c v)^n. 
 \end{equation}
 
  We establish the following generalization of \eqref{eq: MA contact}:
 
 \begin{prop}
 	Let $\theta$ be a smooth closed real $(1,1)$-form representing a big and nef class.	Assume $u \in \PSH(X,\theta)$, $v \in \cE(X,\theta)$, $u\leq v$ and $(\theta+dd^c v)^n$ has finite entropy. Then 
	\[
	{\bf 1}_{\{u=v\}} (\theta +dd^c u)^n = {\bf 1}_{\{u=v\}} (\theta +dd^c v)^n. 
	\]
 \end{prop}
  Here $(\theta+dd^c u)^n$ is the non-pluripolar Monge-Amp\`ere measure introduced in \cite{BEGZ10}.

 
 \medskip
  
 \noindent {\bf Acknowledgement.}    We warmly thank Tam\'as Darvas and L\'aszl\'o Lempert for reading the first version of this paper and giving several useful comments.


\section{Preliminaries}\label{sec:prelim}

In the whole paper, $(X,\omega)$ is a compact K\"ahler manifold of dimension $n\in \bN^*$.
We assume $\omega$ is normalized so that $\int_X \omega^n=1$. In this section we recall and extend several classical facts in pluripotential theory that will be used in the sequel. 

\subsection{Quasi-plurisubharmonic functions}

Let $\theta$ be a smooth closed real $(1,1)$-form on $X$. A function $u:X \rightarrow \mathbb{R}\cup \{-\infty\}$ is quasi-plurisubharmonic (qpsh) if  locally  $u=\rho +\varphi$ where $\varphi$ is plurisubharmonic (psh) and $\rho$ is smooth. A qpsh function $u$ is $\theta$-psh if $\theta+dd^c u \geq 0$ in the weak sense of currents.  We let $\PSH(X,\theta)$ denote the class of all $\theta$-psh functions on $X$ which are not indentically $-\infty$. The cohomology class $\{\theta\}$ is nef if $\{\theta+\varepsilon \omega\}$ is K\"ahler for all $\varepsilon>0$, and it is big if $\PSH(X,\theta-\varepsilon \omega)$ is not empty for some $\varepsilon>0$. We say that $\theta$ is a big form if its cohomology class $\{\theta\}$ is big.

Throughout this paper we assume that $\{\theta\}$ is big. By Demailly's regularization theorem  \cite{Dem92,Dem94} there exists  $\psi \in \PSH(X,\theta-\varepsilon \omega)$ with analytic singularities, which is smooth in some Zariski open set called the ample locus of  $\{\theta\}$, and denoted by ${\rm Amp}(\theta)$.  Here, a function $f: X\rightarrow \bR \cup \{-\infty\}$ has analytic singularities if it can be written locally as 
\[
f(x) = c \log \left ( \sum_{j=1}^N |f_j|^2 \right ) +  \rho,
\] 
where $c>0$,  $\rho$ is smooth and the functions $f_j$ are holomorphic. 

Given two $\theta$-psh functions $u$ and $v$  we say that $u$ is more singular than  $v$, and we write $u\preceq v$,  if $u\leq v+C$ for some constant  $C$. These two functions have the same singularities, and we write   $u\simeq v$, if $u\preceq v$ and $v\preceq u$. The envelope 
\[
V_{\theta}:= \sup \{u \in \PSH(X,\theta)\; : \; u\leq 0 \}
\]
is a $\theta$-psh function with minimal singularities, it is locally bounded and continuous in $\Amp(\theta)$.

The classical Monge-Amp\`ere capacity (see \cite{BT82}, \cite{Kol98}, \cite{GZ05}) is defined by 
\[
\capa_{\omega}(E) := \sup \left \{ \int_E (\omega+dd^c u)^n \; : \; u\in \PSH(X,\omega),\; -1\leq u \leq 0 \right \}. 
\]
A sequence $u_j$ converges in capacity to $u$ if for any $\varepsilon>0$ we have
\[
\lim_{j\to +\infty} \capa_{\omega}(\{|u_j-u|\geq \varepsilon\}) =0. 
\]
Any quasi-psh function $u$ is quasi-continuous (see \cite{GZbook}, \cite{Kli91}) : for each $\varepsilon>0$ there exists an open set $U$ with $\capa_{\omega}(U)<\varepsilon$ such that $u$ restricted on $X\setminus U$ is continuous. 

A set $E$ is quasi-open (respectively quasi-closed) if for any $\varepsilon>0$ there exists an open  (respectively closed) set $U$ such that  $\capa_{\omega} (E\setminus U) \leq \varepsilon$ and $\capa_{\omega} (U\setminus E) \leq \varepsilon$.  
The weak convergence of measures will be denoted by  $\rightharpoonup$. If $\mu_j\rightharpoonup \mu$, $U$ is open and $V$ is closed then 
$$
\liminf_j \mu_j(U) \geq \mu(U), \qquad \limsup_j \mu_j(V) \leq \mu(V).
$$ 
By the following lemma we can replace the open set $U$ by any quasi-open set and the closed set $V$ by any quasi-closed set, provided that the measures $\mu_j,\mu$ are uniformly dominated by capacity.  
\begin{lemma}\label{lem: convergence quasi open}
	Let $(\mu_j)$ be a sequence of positive Radon measures such that $\mu_j \rightharpoonup \mu$ for some positive measure $\mu$. Assume there exists a continuous function $F: [0,1] \rightarrow [0,+\infty)$ with  $F(0)=0$ such that $\mu_j + \mu \leq F(\capa_{\omega})$.
	Then for any quasi-open set $E$ and quasi-closed set $F$ we have 
	\[
	\liminf_{j\to +\infty}  \mu_j(E)  \geq \mu(E)\; \text{and}\; \limsup_{j\to +\infty} \mu_j(F) \leq \mu(F). 
	\]
\end{lemma}

\begin{lemma}
	\label{lem: convergence bounded}
			Let $(\mu_j)$ be a sequence of positive Radon measures such that $\mu_j \rightharpoonup \mu$ for some positive measure $\mu$. Assume there exists a continuous function $F: [0,1] \rightarrow [0,+\infty)$ with  $F(0)=0$ such that $\mu_j + \mu \leq F(\vol_{\omega})$. If a sequence of Borel functions $(u_j)$ is uniformly bounded and converges in $L^1$ to $u$ then 
		$$\int_X u_j d\mu_j \to \int_X u d\mu.$$ 
\end{lemma}

\begin{lemma}
	\label{lem: conv measure}
	Assume $(\mu_j)$ and $\mu$ are as in Lemma \ref{lem: convergence quasi open}. Assume $(f_j)$ is a sequence of uniformly bounded quasi-continuous functions converging in capacity to a bounded quasi-continuous function $f$. Then $f_j \mu_j \rightharpoonup f\mu$. 
\end{lemma} 

The proof of Lemma \ref{lem: convergence quasi open}, Lemma \ref{lem: convergence bounded} and Lemma \ref{lem: conv measure}  are very similar to that of \cite[Corollary 2.9]{DDL1}. 

\begin{lemma}\label{lem: entropy vol}
	Assume $\mu$ is a positive Radon measure with finite entropy $\Ent(\omega^n,\mu) \leq C$. Then $\mu\leq F(\vol_{\omega})\leq F(\capa_{\omega})$ for some function $F$ as in Lemma \ref{lem: convergence quasi open} which depends only on $C$. 
\end{lemma}
\begin{proof}
	Take a Borel set $E\subset X$ with $\vol(E):=\int_E \omega^n\in (0,1]$ and set $A:=\vol(E)^{-1/2}$. A direct computation shows that 
	\begin{eqnarray*}
	\mu(E) &= &  \int_E f\omega^n = \int_{E\cap \{f>A\}} f\omega^n +\int_{E\cap \{f \leq A\}}  f\omega^n \\
	&\leq & \frac{1}{\log A} \int_X f \log f \; \omega^n + A \vol(E)\\
	&\leq & C (-\log (\vol(E)))^{-1} + \vol(E)^{1/2}.
	\end{eqnarray*}
Noting also that $\vol_{\omega}\leq  \capa_{\omega}$, the result follows with $F(t):= C (-\log t)^{-1} + \sqrt{t}$ (that is increasing). 
\end{proof}

Assume $\theta_1,...,\theta_p$, $p\leq n$, are smooth closed real $(1,1)$-forms representing big cohomology classes. Let $u_j \in \PSH(X,\theta_j)$, for $j=1,...,p$. If these functions are locally bounded in the corresponding ample loci then by Bedford-Taylor \cite{BT76,BT82} we can define the Monge-Amp\`ere product
\[
(\theta_1+dd^c u_1) \wedge ... \wedge (\theta_p+dd^c u_p),
\]that is a positive $(p,p)$-current on $\cap_{j=1}^n {\rm Amp}(\theta_j)$. Since the total mass of this current is finite, one can extend it trivially over $X$. As shown by Bedford-Taylor \cite{BT87}, the Monge-Amp\`ere product is local in the plurifine topology: if $U$ is a quasi-open set and $u_j=v_j$ on $U$ for $j=1,...,p$, then 
\[
{\bf 1}_U (\theta_1+dd^c u_1) \wedge ... \wedge (\theta_p+dd^c u_p) = {\bf 1}_U (\theta_1+dd^c v_1) \wedge ... \wedge (\theta_p+dd^c v_p). 
\]We next define the non-pluripolar product  following \cite{BEGZ10}. For each  $t>0$, we consider  $u_{j,t}:= \max(u_j,V_{\theta_j}-t)$ and $U_t:= \cap_{j=1}^n \{u_j>V_{\theta_j} -t\}$. By locality of the Monge-Amp\`ere product with respect to the plurifine topology, the sequence of measures 
\[
{\bf 1}_{U_t} (\theta_1+dd^c u_{1,t})\wedge ... \wedge (\theta_p+dd^c u_{p,t})
\]is increasing in $t$.  The limit as $t\to +\infty$, denoted by $(\theta_1+dd^c u_1)\wedge ... \wedge (\theta_p+dd^c u_p)$, is a positive current on $X$. By Sibony \cite{Sib85} (see also \cite{Dembook}, \cite{BEGZ10}), this current is closed.    In case when $u_1=...=u_n=u$ and $\theta_1=...=\theta_n=\theta$, we obtain the non-pluripolar Monge-Amp\`ere measure of $u$ denoted by $(\theta+dd^c u)^n$ or simply by $\theta_u^n$. 
By construction, $\theta_u^n$ is a positive Radon measure on $X$ with total mass satisfying 
\[
\int_X \theta_u^n \leq \vol(\theta) := \int_X \theta_{V_{\theta}}^n. 
\]
We let $\cE(X,\theta)$ denote the set of all $u\in \PSH(X,\theta)$ with full non-pluripolar Monge-Amp\`ere mass: $\int_X \theta_u^n =\vol(\theta)$. For each $p>0$ we let $\cE^p(X,\theta)$ denote the set of all $u\in \cE(X,\theta)$ such that $E_p(u):= \int_X |u-V_{\theta}|^p \theta_u^n <+\infty$. 

We recall the following global version of the local maximum principle of Bedford-Taylor: 
\begin{lemma}
	\label{lem: max principle}
	Assume $u$ and $v$ are in $\PSH(X,\theta)$. Then 
	\[
	(\theta+dd^c \max(u,v))^n \geq {\bf 1}_{\{u\geq v\}} (\theta+dd^c u)^n + {\bf 1}_{\{u<v\}} (\theta+dd^c v)^n. 
	\]
	 If in addition $u\leq v$, then 
	\[
	{\bf 1}_{\{u=v\}} (\theta+dd^c u)^n \leq {\bf 1}_{\{u=v\}} (\theta+dd^c v)^n. 
	\]
\end{lemma}
For a proof of this lemma we refer to \cite[Lemma 4.5]{DDL5}. 
We end this section with the following lemma:
\begin{lemma}\label{MA_ep}
Let $v\in \cE(X, \theta)$ and $u\in \cE^p(X, \theta)$. Assume $\theta_v^n \leq C\theta_u^n $, for some constant $C>0$. Then $v \in \cE^p(X, \theta)$.
\end{lemma} When $\{\theta\}$ is a K\"ahler class this result was proved by Guedj-Zeriahi in  \cite{GZ07}. Their proof can be easily adapted in the setting of big cohomology classes (see \cite{DV21} for an extension of this result).  We refer to \cite{DNL22} for a different approach. 
\subsection{Quasi-plurisubharmonic envelopes}
For each Lebesgue measurable function  $f$ on $X$, we define
\[
P_{\theta}(f):= \left(\sup \{u\in \PSH(X,\theta) \; : \; u\leq f \; \text{quasi everywhere on}\; X \}\right)^*. 
\]
Here, quasi everywhere means outside a pluripolar set, i.e. a set contained in the $-\infty$-locus of some $v\in \PSH(X,\theta)$. If  $f= \min(u,v)$, we will simply write $P_{\theta}(\min(u,v))=P_{\theta}(u,v)$.

If $f$ is lower semicontinuous (lsc), one can show,  by a balayage process, that the Monge-Amp\`ere measure $(\theta + dd^c P_{\theta}(f))^n$ is supported on the contact set $\{P_{\theta}(f)=f\}$. For more general obstacles $f$,  the following  was proved in \cite{GLZ19}:
\begin{lemma}\label{lem: MA contact}
	If $f$ is quasi lower semicontinuous on $X$ and $P_{\theta}(f) \in \PSH(X,\theta)$, then the non-pluripolar Monge-Amp\`ere measure $(\theta+dd^c P_{\theta}(f))^n$ is supported on the contact set $\{P_{\theta}(f)=f\}$. Moreover, if  $f=\min(u,v)$ with $u,v\in \PSH(X,\theta)$, then  
	\[
	(\theta+dd^c P_{\theta}(u,v))^n \leq {\bf 1}_{\{P_{\theta}(u,v)=u\}} \theta_u^n + {\bf 1}_{\{P_{\theta}(u,v)=v\}} \theta_v^n. 
	\]
\end{lemma}
Here we say that $f$ is quasi lower semicontinuous if for each $\varepsilon>0$ there exists an open set $U$ of small capacity $\capa_{\omega}(U)<\varepsilon$, such that $f|_{X\setminus U}$ is lower semicontinuous. If $f$ is quasi lsc, there exists a decreasing sequence of lsc functions converging to $f$ quasi everywhere. 
We underline that in \cite{GLZ19}, it was assumed that $f$ is bounded but the general result as stated above follows easily by approximation.

We also stress that the regularity of $f$ in Lemma \ref{lem: MA contact} is necessary. As a counterexample, we take $f$ to be $0$ in some small open ball $B$ and  $+\infty$ on $X\setminus B$. If $(\theta+dd^c P_{\theta}(f))^n$ is supported on the contact set $D=\{P_{\theta}(f) =f\}\subset B$ then $\int_X (\theta+dd^c P_{\theta}(f))^n \leq \int_B \theta^n$.  But $P_{\theta}(f)\simeq V_{\theta}$ because $V_{\theta}\leq P_{\theta}(f) \leq V_{\theta,B} \leq C$. We recall that 
$$ V_{\theta,B}:= \sup\{ u\in \PSH(X, \theta) \,:\, u\leq 0 \; {\rm in }\; B\}, $$
thus $P_\theta(f) \leq  V_{\theta,B}$ by definition. So the total mass $\int_X (\theta+dd^c P_{\theta}(f))^n$ must be the volume of  $\{\theta\}$.

\begin{lemma}
	\label{lem: envelope big}
	Assume $\{\eta\}$, $\{\theta\}$ are big classes and $\eta \geq \theta$. If $u\in \cE(X,\eta)$ then $P_{\theta}(u) \in \cE(X,\theta)$. If moreover $ u\in \cE^p(X,\eta)$ then $P_{\theta}(u) \in \cE^p(X,\theta)$.
\end{lemma}

\begin{proof}
	Without loss of generality we can assume that $u\leq 0$. The proof of the first statement is identical to that of \cite[Lemma 2.7]{Lu21}. We next prove the last one. Since $\theta\leq \eta$, then $V_\theta\leq V_\eta$ and $(\theta+dd^c P_\theta(u))^n$ is supported on the contact set $\mathcal{C}:= \{P_{\theta}(u)=u\}$ (Lemma \ref{lem: MA contact}), we have
	\[
	\int_X |P_\theta(u)-V_{\theta}|^p (\theta+dd^c P_{\theta}(u))^n \leq  \int_{\mathcal{C}}|u-V_{\eta}|^p(\eta+dd^c P_{\theta}(u))^n.
	\]
	Now, Lemma \ref{lem: max principle} ensures that 
	\[
	\int_{\mathcal{C}}|u-V_{\eta}|^p(\eta+dd^c P_{\theta}(u))^n \leq \int_{\mathcal{C}}|u-V_{\eta}|^p(\eta+dd^c u)^n<+\infty.
	\]
	This proves the last statement. 
\end{proof}

 \subsection{Plurisubharmonic geodesics}
Thanks to the discovery of Semmes and Donaldson, the geodesic equation can be understood in the weak sense of measures using pluripotential theory. An advantage of this interpretation is that it also applies in the context of big cohomology classes in which case an approximation with K\"ahler classes is impossible.  We next recall the construction in \cite{DDL1} following Berndtsson's idea \cite{Bern15}. 

For a curve $[0,1] \ni t \mapsto u_t \in \PSH(X,\theta)$  we define 
\begin{equation}\label{eq: complexified curve}
X\times D \ni (x,z) \mapsto U(x,z) := u_{\log |z|}(x),	
\end{equation}where  $D= \{z\in \bC, \; 1 < |z| < e \}$ and $\pi: X\times D \rightarrow X$ is the projection on the first factor. 
\begin{definition}
	We say that $t\mapsto u_t$ is a subgeodesic if $(x,z) \mapsto U(x,z)$ is a $\pi^{*}\theta$-psh function on $X\times D$.
\end{definition}

\begin{definition}
	For  $\varphi_0,\varphi_1 \in \PSH(X,\theta)$, we let $\mathcal{S}_{[0,1]}(\varphi_0,\varphi_1)$ denote the set of  all subgeodesics $[0,1] \ni t \mapsto u_t$ such that $\limsup_{t\to 0} u_t\leq \varphi_0$ and $\limsup_{t\to 1} u_t\leq \varphi_1$. 
\end{definition}

Let $\varphi_0,\varphi_1 \in \PSH(X,\theta)$. For $(x,z)\in X\times D$ we define
$$
\Phi(x,z) :=  \sup \{ U(x,z) \; :\;  U \in \mathcal{S}_{[0,1]}(\varphi_0,\varphi_1) \}. 
$$
The curve $t\mapsto \varphi_t$ constructed from $\Phi$ via \eqref{eq: complexified curve} is called the plurisubharmonic (psh) geodesic segment connecting $\varphi_0$ and $\varphi_1$.  

Geodesic segments connecting two general $\theta$-psh functions may not exist. If $\varphi_0, \varphi_1 \in \mathcal{E}^p(X,\theta)$, it was shown in \cite[Theorem 2.13]{DDL1} that $P(\varphi_0,\varphi_1) \in \mathcal{E}^p(X,\theta)$. Since $P(\varphi_0,\varphi_1) \leq \varphi_t$, we obtain that $t \to \varphi_t$ is a curve in $\mathcal{E}^p(X,\theta)$. Also, each subgeodesic segment is convex in $t$:  
\[
\varphi_t\leq \left (1-t\right )\varphi_0 + t\varphi_1, \ \forall t\in [0,1]. 
\]
Consequently the upper semicontinuous regularization (with respect to both variables $x,z$) of $\Phi$ is again in $\mathcal{S}_{[0,1]}(\varphi_0,\varphi_1)$, hence so is $\Phi$. If $\varphi_0,\varphi_1$ have the same singularities then the geodesic $\varphi_t$ is Lipschitz on $[0,1]$ (see \cite[Lemma 3.1]{DDL1}): 
\begin{equation}
	\label{eq: Lip}
	|\varphi_t-\varphi_s| \leq |t-s| \sup_{X} |\varphi_0-\varphi_1|, \ \forall t,s \in [0,1]. 
\end{equation}

The Darvas $d_1$ metric has a very special property allowing to express $d_1(u,v)$ exclusively in terms of the Monge-Amp\`ere energy: 
\[
d_1(u,v) = E(u) +E(v)- 2 E(P(u,v)), \qquad u,v\in \cE^1(X,\theta),
\]
where $E(u)=E(\theta;u,V_{\theta})$ is defined in the next subsection. 
In \cite{DDL3} we take this as the definition of $d_1$ for potentials in a big cohomology class $\theta$, and prove that $(\cE^1(X,\theta), d_1)$ is a complete geodesic metric space.  


 \subsection{The Mabuchi K-energy}\label{K-en}
 
  The scalar curvature of a K\"ahler metric  $\omega$ is the trace of its Ricci form: 
\[
{\rm Scal}(\omega) := n \frac{{\rm Ric}(\omega) \wedge \omega^{n-1}}{\omega^n} \in C^{\infty}(X,\bR). 
\]
A K\"ahler metric $\omega$ is cscK (constant scalar curvature)  if the scalar curvature  ${\rm Scal}(\omega)= \bar{S}$ is constant. A simple application of Stokes theorem ensures that this constant depends only on the first Chern class  and $\{\omega\}$: 
\[
\bar{S} = \bar{S}_{\omega} = n\vol(\omega)^{-1}  c_1(X) \cdot \{\omega\}^{n-1}, 
\]
where $\vol(\omega)= \int_X \omega^n$ is the volume of $\omega$. 

The K-energy $\cM_\omega$, first introduced by Mabuchi \cite{Mab87}, is defined so that cscK metrics are critical points. Chen and Tian have then found  an explicit formula for $\cM_\omega$ that we now recall.  The Monge-Amp\`ere energy (or Aubin-Mabuchi energy) is defined as 
\begin{equation*}
E(\omega; u, v)= \frac{1}{(n+1)\vol(\omega)} \sum_{j=0}^{n} \int_X (u-v) \omega_u^{j} \wedge \omega_v^{n-j}\quad u,v \in \cE^1(X,\omega). 
\end{equation*}
Given a closed smooth $(1,1)$-form $\chi$, the $\chi$-contracted Monge-Amp\`ere energy is defined as
$$
 E_{\chi}(\omega; u,v) = \frac{1}{n\vol(\omega)}\sum_{j=0}^{n-1}\int_X (u-v) {\chi} \wedge \omega_u^{j} \wedge \omega_v^{n-1-j}\quad  u,v \in \cE^1(X,\omega).
$$
The Mabuchi K-energy (w.r.t. $\{\omega\}$) is defined as 
$$
\cM_{\omega}(u) := \bar{S}_{\omega} E(\omega;u,0) - n E_{\Ric(\omega)}(\omega;u,0) + \Ent(\omega^n,\omega_u^n), \; u \in \cE^1(X,\omega), 
$$
where given two positive Radon measures $\mu, \nu$, the relative entropy $\Ent(\nu,\mu)$ is defined as  
\[
\Ent(\nu,\mu) := \int_X \log\left (\frac{d\mu}{d\nu}\right) d\mu,
\]
if $\mu$ is absolutely continuous with respect to $\nu$, and $+\infty$ otherwise. 

Observe that, choosing a different closed real $(1,1)$-form representing the K\"ahler class, $\omega_v\in \{\omega\}$ (not necessarily a K\"ahler form), any $\omega$-psh function $u$ can be written as $u=\f-v$ where $\f\in \PSH(X, \omega_v)$ and $\cM_\omega$ can be then re-written as
\begin{equation}\label{mabuchiNotK}
\cM_{\omega}(\f) := \bar{S}_{\omega} E(\omega_v;\f,v) - n E_{\Ric(\omega)}(\omega_v;\f,v) + \Ent(\omega_v^n,\omega_\f^n), \quad \f \in \cE^1(X,\omega_v).
\end{equation}

 Mabuchi \cite{Mab86,Mab87}  has proved that $\cM_{\omega}$ is convex along smooth geodesics in $\cH$.  Smooth geodesic rays can be constructed using the flow of real holomorphic vector fields \cite[Theorem 3.5]{Mab87}. Unfortunately, even if  $u_0$ and $u_1$ are in $\cH$, the weak geodesic segment $(u_t)_{t \in [0,1]}$ is not necessarily in $\cH$.  In a breakthrough paper of Berman-Berndtsson \cite{BB17}  (see also \cite{CLP16} for a different proof) the authors showed that  $t \mapsto \cM(u_t)$ is convex and continuous on $[0,1]$ if $u_0$ and $u_1$ are in $\cH$.  Building on the explicit formula of Chen and Tian, it was shown in \cite{BDL17} that the extended Mabuchi K-energy  $\cM_{\omega}$  is  convex and continuous in $[0,1] $ along psh geodesics in $\cE^1(X,\omega)$.  
 
 \subsection{Finite entropy potentials}
Let $\theta$  be a closed smooth real $(1,1)$-form representing a big cohomology class. We let $\Ent(X,\theta)$ denote the set of all $u \in \cE(X,\theta)$ with $\Ent(\omega^n, \theta_u^n) <+\infty$. In confirming a conjecture of Aubin \cite{Aub84}, it was proved in \cite{DGL20} that $\Ent(X,\omega) \subset \cE^{\frac{n}{n-1}}(X,\omega)$. As will be shown below, one can extend these results for big cohomology classes.

  \begin{theorem}
 	\label{thm: Moser-Trudinger inequality} 
 	Fix $p >0$. Then there exist $c>0$, $C>0$ depending on $X$, $\theta$, $n$, $p$ such that, for all $\varphi\in \mathcal{E}^p(X,\theta)$ with $\sup_X \varphi=-1$,  we have
 	\[
 	\int_X \exp \left ( c  |E_p(\varphi)|^{-\frac{1}{n}} |\varphi-V_{\theta}|^{1+\frac{p}{n}}\right )\omega^n \leq C.
 	\]
 \end{theorem}
  \begin{proof}
 	By approximation we can assume that $\varphi$ has minimal singularities. For notational convenience we set $q=\frac{n+p}{n}>1$, $\psi:=-a(-\varphi+V_{\theta})^q$ and $u := P_{\theta}(\psi+V_{\theta})$, where $a>0$ is a small constant suitably chosen below. Observe that, since $\varphi$ has minimal singularities, $\psi$ is bounded and as consequence $u$ has minimal singularities as well. Moreover, by construction $u\leq V_\theta$.
 	Note also that $E_p(\varphi) = \int_X(-\varphi+V_{\theta})^p \theta_{\varphi}^n \geq \vol(\theta)$ because $\varphi\leq V_{\theta}-1$. For simplicity we assume that $\vol(\theta)=1$.   A direct computation in $\Omega$, the ample locus of $\theta$, shows that the function $v:=-a^{-1/q}(V_{\theta}-u)^{1/q} + V_{\theta}$ is $\theta$-psh with minimal singularities and 
 	\begin{flalign*}
 		(\theta+dd^c v) &\geq \left (1- a^{-1/q}q^{-1} (V_{\theta}-u)^{(1-q)/q} \right ) (\theta+dd^c V_{\theta})\\
 		&+ a^{-1/q}q^{-1} (V_{\theta}-u)^{(1-q)/q} (\theta+dd^c u). 
 	\end{flalign*}
 	We also have that $v \leq  \varphi$ with equality on the contact set $\mathcal{C}:= \{u= \psi+ V_{\theta}\}$. By \cite[Lemma 4.5]{DDL5} we have 
 	\begin{equation}
 		\label{eq: DGL contact 1}
 		{\bf 1}_{\mathcal{C}}(\theta+dd^c v)^n\leq {\bf 1}_{\mathcal{C}}(\theta+dd^c \f)^n.
 	\end{equation}
 	 Also, since $\psi+ V_{\theta}$ is quasi continuous, by Lemma \ref{lem: MA contact} the non-pluripolar Monge-Amp\`ere measure $(\theta+dd^c u)^n$ is supported on $\mathcal{C}$.
 We set 
\[
G:=\Omega \cap \left\{|u-V_{\theta}| > a^{-\frac{n}{p}} q^{-\frac{n+p}{p}}\right\}.
\] 
Observe that in the open set $G$, we have $\left (1- a^{-1/q}q^{-1} (V_{\theta}-u)^{(1-q)/q} \right )\geq 0$, hence
\begin{eqnarray}
\nonumber	(\theta+dd^c v) &\geq & a^{-1/q}q^{-1}(V_{\theta}-u)^{(1-q)/q} (\theta+dd^c u)\qquad {\rm on }\; G\\
	\label{eq: DGL contact 2}	 &= & a^{-1}q^{-1}(V_{\theta}-\varphi)^{(1-q)} (\theta+dd^c u)\qquad \quad\;\, {\rm on }\; G \cap \mathcal{C},
\end{eqnarray}
where the last equality follows from the fact that $V_{\theta}-u= a(V_{\theta}-\varphi)^q$ on $\mathcal{C}$. Combining this with \eqref{eq: DGL contact 1} and \eqref{eq: DGL contact 2} we obtain
\begin{flalign*}
{\bf 1}_{G\cap \mathcal{C}} \;a^{-n}q^{-n}(V_{\theta}-\varphi)^{n(1-q)} (\theta+dd^c u)^n
	&\leq {\bf 1}_{G\cap \mathcal{C}}(\theta+dd^c v)^n \leq {\bf 1}_{G\cap \mathcal{C}}(\theta+dd^c \varphi)^n,
\end{flalign*}
which is equivalent to 
\begin{equation}
	\label{eq: DGL contact 3}
		{\bf 1}_{G}  (\theta+dd^c u)^n= {\bf 1}_{G\cap \mathcal{C}}  (\theta+dd^c u)^n\leq {\bf 1}_{G\cap \mathcal{C}}\; a^nq^n(V_{\theta}-\varphi)^{p}(\theta+dd^c \f)^n.
\end{equation}
We now choose $a\in(0,1)$ so that 
\[
2a^n q^n \int_X |\varphi-V_{\theta}|^p (\theta+dd^c \varphi)^n=2a^n q^n E_p(\varphi) = 1.
\]
Integrating over $X$ both sides of \eqref{eq: DGL contact 3}, we obtain
\[
\int_G (\theta+dd^c u)^n \leq \frac{1}{2}.
\]
It thus follows that $G\neq X$, or equivalently that $G^c\neq \emptyset$. In particular $\sup_X u =\sup_X (u-V_\theta) \geq b:=-a^{-\frac{n}{p}} q^{-\frac{n+p}{p}}$. Let $c_0$, $C_0>0$ be uniform constants such that, for all $\phi \in \PSH(X,\theta)$ with $\sup_X \phi=0$ we have 
\[
\int_X e^{c_0|\phi|} \omega^n \leq C_0.
\]
We now argue exactly the same as in \cite[Theorem 2.1]{DGL20} to obtain the result with 
$
c= 2^{-1-1/n}q^{-1}c_0,\, {\rm{and}} \, C= C_0(e^{c_0}+1). 
$
\end{proof}

\begin{theorem}
	\label{thm: Entropy implies Ep and integrability} 
	Fix $B>0$ and set $p=\frac{n}{n-1}$. 
	There exist $c,C>0$ depending on $B,X,\theta,\omega,n$ such that for all $\varphi\in \Ent_B(X,\theta)$ we have 
	\[
	\int_X e^{c (-\varphi+V_{\theta})^p} \omega^n \leq C\, \quad \text{and}\,\quad E_p(\varphi) \leq C. 
	\]
	In particular $\Ent(X,\theta) \subset \mathcal{E}^{\frac{n}{n-1}}(X,\theta)$. 
 \end{theorem}
 Here 
\[
\Ent_B(X,\theta) := \{u \in \mathcal{E}^1(X,\theta) \; : \;  \sup_X u =-1,\quad \Ent(\omega^n,\theta_u^n) \leq B \}.
\]

 \begin{proof}
 	Using Theorem \ref{thm: Moser-Trudinger inequality}, the proof of Theorem \ref{thm: Entropy implies Ep and integrability}  is identical to that of \cite[Theorem 3.4]{DGL20}. 
 \end{proof}
 
 \begin{corollary}\label{cor: BBEGZ entropy}
 	Assume $(u_j)$ is a sequence in $\cE^1(X,\theta)$ such that $$\Ent(\omega^n,(\theta+dd^c u_j)^n)\leq C,$$ for some uniform constant $C$. If $u_j \to u\in \PSH(X,\theta)$ in $L^1$ then $u\in \cE^1(X,\theta)$ and $d_1(u_j,u)\to 0$. 
 \end{corollary}
 \begin{proof}
 We write $(\theta+dd^c u_j)^n = f_j \omega^n$. By assumption we have $ \int_X f_j \log f_j \omega^n\leq C$. Theorem \ref{thm: Entropy implies Ep and integrability} then implies that for any $j,k$

 \begin{equation}\label{eq00ele}
 \int_X |u_j-u_k|^{\frac{n}{n-1}}(f_j+f_k)\, \omega^n \leq C.
 \end{equation}
 Fixing $\varepsilon>0$, by  Egorov theorem there exists an open set $U$ such that $\int_U \omega^n <\varepsilon$ and $u_j$ converge uniformly to $u$ on $X\setminus U$. We then have 
 \[
 \lim_{j,k\to +\infty}\int_{X\setminus U} |u_j-u_k| (f_j+f_k) \omega^n \to 0.
 \]
 On the other hand, by H\"older inequality
 \[
 \int_U |u_j-u_k| (f_j+f_k) \omega^n \leq \left (\int_U |u_j-u_k|^{\frac{n}{n-1}}(f_j+f_k)\omega^n\right)^{\frac{n-1}{n}} \left( \int_U (f_j+f_k) \omega^n \right)^{\frac{1}{n}}. 
 \]
 Now, the first term on the right-hand side is uniformly bounded thanks to \eqref{eq00ele}, while the second term is dominated by $(2F(\varepsilon))^{1/n}\leq 2F(\varepsilon)$, where $F$ is the function in Lemma \ref{lem: entropy vol}. It thus follows that 
 	\[
 	\lim_{j,k\to +\infty} \int_X |u_j-u_k| (f_j+f_k)\, \omega^n =0. 
 	\]
By \cite[Theorem 3.7]{DDL3} the sequence $(u_j)$ is Cauchy in $(\cE^1(X,\theta), d_1)$. Thus, it converges in $d_1$ to some $v\in \cE^1(X,\theta)$ \cite[Theorem 3.10]{DDL3}. But $d_1$-convergence implies $L^1$-convergence (\cite[Theorem 3.11]{DDL3}). We therefore have $v=u\in  \cE^1(X,\theta)$ and $d_1(u_j,u)\to 0$. 
 \end{proof}

\subsection{Stability of solutions} 

Consider a sequence of solutions to Monge-Amp\`ere equations 
\[
(\theta_j +dd^c u_j)^n = f_j \omega^n,\; u_j \in \cE(X,\theta_j), \; \sup_X u_j =0,
\]
where $(f_j)$ is a sequence of densities converging in $L^1$ to some $f$, and $\theta_j$ is a sequence of big forms converging to a big form $\theta$. Naturally, we expect that $u_j\to u$ in $L^1$, where $u$ is the unique solution to 
\[
(\theta+dd^c u)^n = f\omega^n, \; u\in \cE(X,\theta),\; \sup_X u=0. 
\]  

\begin{theorem}\label{thm: stability of solutions}
	Assume $\theta_j$ decreases to $\theta$, and $\int_X f_j\log f_j \; \omega^n \leq C$ is uniformly bounded. 
	Then a subsequence of $(u_{j})$, still denoted by $(u_j)$, can be sandwiched between two monotone sequences 
	\[
	\cE^1(X,\theta)\ni \f_j \leq u_j \leq \psi_j \in \cE^1(X,\theta_j); \; \f_j \nearrow u, \; \text{ and}\;  \psi_j \searrow u.
	\]
	In particular $u_j \to u$ in capacity. 
\end{theorem}
To be more precise, we assume in the above theorem that $\theta + \varepsilon_j \omega \geq \theta_j\geq \theta_{j+1}\geq \theta$, where  $(\varepsilon_j)$ is a sequence of positive real numbers decreasing to $0$. 
Note that by Theorem \ref{thm: Entropy implies Ep and integrability},  we have $u_j \in \cE^1(X,\theta_j)$ and $u\in \cE^1(X,\theta)$. 

\begin{proof}
After extracting a subsequence we can assume that $f_j\to f$ almost everywhere, that $u_j \to \hat{u}\in \PSH(X,\theta)$ in $L^1$ and almost everywhere. In particular $\sup_X \hat{u}=0$. By the proof of \cite[Lemma 2.8]{DDL4} we have 
\[
(\theta+dd^c \hat{u})^n\geq f\omega^n.
\] 
Comparing their total mass we see that the above inequality is an equality. By uniqueness \cite[Theorem A]{BEGZ10}, we thus have $u=\hat{u}$. 
 
For each $j$, we  consider 
 \[
 \psi_j:= \left(\sup_{k\geq j} u_k\right)^*. 
  \]
 Then $\psi_j\in \cE^1(X, \theta_j)$ and $\psi_j \searrow u$. 
To produce the lower sequence sandwiching $(u_j)$, it is natural to consider the envelope 
$$
P_{\theta}\left(\inf_{k\geq j} u_k\right).
$$ 
A priori such a function  could be identically $-\infty$, and this is the reason why we should first consider an appropriate subsequence of $(u_j)$.  
We define 
$$
v_j:= P_{\theta}(u_j),
$$
which, by Lemma \ref{lem: envelope big}, belongs to $\cE^1(X,\theta)$.  By Lemma \ref{lem: MA contact}, we have that $\theta_{v_j}^n$ is supported on the contact set $\mathcal{C}=\{v_j=u_j\}$. Lemma \ref{lem: max principle} thus yields  
	\begin{equation}\label{mass00}
	\theta_{v_j}^n = {\bf 1}_{\mathcal{C}} \theta_{v_j}^n  \leq  {\bf 1}_{\mathcal{C}} (\theta_j +dd^c v_j)^n=  {\bf 1}_{\mathcal{C}} (\theta_j +dd^c u_j)^n\leq  f_j\omega^n. 
	\end{equation}
	By the above and the H\"older-Young inequality we have 
	\[
	-\alpha  \vol(\theta) \sup_X v_j \leq \alpha\int_X (-v_j)(\theta+dd^c v_j)^n \leq \int_X e^{-\alpha v_j} \omega^n + \int_X f_j \log (f_j+1) \omega^n \leq C. 
	\]
	Here $\alpha>0$ is a uniform constant ensuring $\int_X e^{-\alpha \psi} \omega^n \leq C$ for all $\psi\in \PSH(X,\theta)$ with $\sup_X \psi=0$. The existence of $\alpha$ follows from the uniform Skoda integrability theorem, see \cite{GZbook}. 
We then get a uniform bound for $\sup_X v_j$. Extracting a subsequence we can assume that $v_j$ converges in $L^1$ and almost everywhere to some $v \in \PSH(X,\theta)$. 
	By \eqref{mass00}, we have $\Ent(\omega^n, \theta_{v_j}^n)\leq C$. Hence, by Corollary \ref{cor: BBEGZ entropy}, we have $d_1(v_j,v) \to 0$. Now, by the proof of \cite[Theorem 3.10]{DDL3}, after extracting a subsequence, the function
	\[
	\f_{j} := P_{\theta} \left(\inf_{k\geq j} v_k\right),	
	\]	
	belongs to $\in \cE^1(X,\theta)$, and $\f_j \nearrow v$. 
		
	To complete the proof, we finally show that $v =u$. Recall that by construction we have $v \leq u$. Since $v_j \geq \f_j$ and $\f_j \nearrow v$, we infer that $v_j \to v$ in capacity. Moreover, since 
	$$\int_X (\theta+dd^c v_j)^n = \int_X (\theta+dd^c v)^n=\vol(\theta),$$ by \cite{DDL2} we therefore have that 
	\[
	(\theta+dd^c v_j)^n \weak  (\theta+dd^c v)^n.
	\]
	 It thus follows from \eqref{mass00} that $\theta_v^n \leq f\omega^n$. Since the two measures have the same total mass, we infer that $\theta_v^n = f\omega^n$. 
	 By uniqueness of solutions \cite{BEGZ10}, we have $v = u+C_0$, for some constant $C_0\leq 0$ because $v\leq u$, and it remains to show that $C_0=0$. 
	 
	Since $(\theta+dd^c v_j)^n$ is supported on the contact set $\{v_j=u_j\}$ (Lemma \ref{lem: MA contact}), we have 	\begin{equation}\label{conv1}
	\int_X e^{v_j} (\theta+dd^c v_j)^n  = \int_X e^{u_j} (\theta+dd^c v_j)^n.  
	\end{equation}
	Since $0\leq e^{u_j}, e^{v_j}\leq 1$ are uniformly bounded, by Lemma \ref{lem: convergence bounded}  we have 
	\[
	\int_X e^{u_j} \theta_{v_j}^n \to \int_X e^{u} \theta_v^n\; \text{and}\; \int_X e^{v_j} \theta_{v_j}^n \to \int_X e^{v} \theta_v^n. 
	\]
	From this and \eqref{conv1} we infer 
	\[
	\int_X e^v \theta_v^n= \int_X e^u \theta_v^n
	\] 
	Since $v=u+C_0$, from the above  we obtain $C_0=0$.
		
	To conclude, we have constructed two sequences $\f_j \nearrow u$, $\psi_j \searrow u$ which satisfy 
	\begin{eqnarray*}
	\cE^1(X, \theta) \ni \f_j \leq u_j\leq  \psi_j\in \cE^1(X, \theta_j).
	\end{eqnarray*}
	In particular $u_j\to u$ in capacity. 
	\end{proof}

 \section{Geodesic distance in K\"ahler classes}

 In this section we prove Theorem \ref{thm: geodesic distance intro} of the Introduction.  \begin{lemma}
 	\label{lem: geodesic distance}
 	Fix $u_0,u_1\in  \Ent(X,\omega)$, and let $(u_t)_{t\in [0,1]}$ be the psh geodesic segment connecting $u_0$ and $u_1$. Assume also that $u_0-u_1$ is bounded on $X$. Then for all $t\in (0,1)$, $\dot{u}_t^-=\dot{u}_t^+$ almost everywhere on $X$ with respect to $(\omega+dd^c u_t)^n$. 
 \end{lemma}
 \begin{proof}
 Let $U$ be the set of all  $x\in X$ with $P_{\omega}(u_0,u_1)(x) >-\infty$. Note that, since $u_t\geq P_{\omega}(u_0,u_1)$, we have $u_t(x)\neq -\infty$ for any $x\in U$. Observe as well that, by Theorem \ref{thm: Entropy implies Ep and integrability}, $u_0, u_1 \in \cE^1(X, \omega)$, and $u_t \geq P_\omega(u_0, u_1) \in \cE^1(X, \omega)$ for any $t\in [0,1]$. This means that the energy $E(u_t)$ of $u_t$ is uniformly bounded, which also implies that $E_{\Ric(\omega)}(u_t)$ is uniformly bounded (see \cite[Lemma 2.7 and Proposition 2.8]{BBGZ13}). 
 
  Also, for $x\in U$, the left and right derivatives $\dot{u}_t^{\pm}(x)\in \bR$ are well-defined thanks to convexity of $t\mapsto u_t(x)$. Since for such $x$, $t\mapsto u_t(x)$ is Lipschitz in $[0,1]$, these two directional derivatives coincide for almost all $t\in (0,1)$. By Fubini's theorem we can find a subset $I\subset [0,1]$ of full Lebesgue measure such that 
 \[
 \dot{u}_t^{+}(x) = \dot{u}_t^{-}(x), \qquad \forall t\in I\; {\rm and \; almost \; every}\; x\in X .
 \]
 Since the Monge-Amp\`ere energy $E=E(\omega;\cdot,0)$ is affine along psh geodesics, we have that for $t\in (0,1)$ and $h\in \bR$ small enough, 
\[
E(u_1)-E(u_0)= \frac{E(u_{t+h})-E(u_t)}{h}. 
\]
Also, since $E$ is also concave along affine paths, for $h>0$ small enough, we  have
\[
\int_X \frac{u_{t}-u_{t-h}}{h} \omega_{u_t}^n \leq E(u_1)-E(u_0) \leq \int_X \frac{u_{t+h}-u_{t}}{h} \omega_{u_t}^n. 
\]
Letting $h\to 0^+$ we obtain 
\begin{equation}
	\label{eq: slope MA energy}
	\int_X \dot{u}_{t}^-\omega_{u_t}^n \leq E(u_1)-E(u_0) \leq \int_X  \dot{u}_t^+ \omega_{u_t}^n. 
\end{equation}
We note here that $X\setminus U$ is a pluripolar set which is negligible with respect to the non-pluripolar Monge-Amp\`ere measure $\omega_{u_t}^n$. 
The convexity of $\cM_{\omega}$ along $u_t$ \cite{BB17} gives 
$$ \cM_{\omega}(u_t) \leq (1-t)\cM_{\omega}(u_0) +t \cM_{\omega}(u_1).$$
By definition, $\cM_\omega$ is the sum of the entropy and an energy part and by assumption $\cM_{\omega}(u_0),  \cM_{\omega}(u_1)\leq C$ and we already observed that the energy part of $u_t$ is uniformly bounded. This implies that 
\begin{equation}\label{unif_entropy}
\Ent(\omega^n, \omega_{u_t}^n)\leq C
\end{equation}
for some uniform constant $C>0$. It thus follows that $(\omega+ dd^c u_t)^n$ is absolutely continuous with respect to the Lebesgue measure. Hence
\[
\dot{u}_t^+ = \dot{u}_t^-.
\]
for $t\in I$ and almost everywhere on $X$ with respect to $\omega_{u_t}^n$. This allows to conclude that in \eqref{eq: slope MA energy} all inequalities are in fact equalities: 
\begin{equation}
	\label{eq: slope MA energy 1}
	\int_X \dot{u}_{t}^-\omega_{u_t}^n =  E(u_1)-E(u_0) = \int_X  \dot{u}_t^+ \omega_{u_t}^n, \; t\in I.
\end{equation}
Now, fix $t \in (0,1)$ and $s \in I$ such that $s>t$. Since $t\mapsto u_t(x)$ is convex and \eqref{eq: slope MA energy 1} holds for $s$, we have that
\[
E(u_1)-E(u_0) = \int_X \dot{u}_{s}^{+} (\omega +dd^c u_{s})^n   \geq \int_X \dot{u}_{t}^+ (\omega +dd^c u_{s})^n. 
\]
Since $u_0-u_1$ is bounded, by \eqref{eq: Lip} we get that $\dot{u}_t$ is bounded as well. Lemma \ref{lem: convergence bounded} together with Lemma \ref{lem: entropy vol} (which can be applied thank to \eqref{unif_entropy}) give
\[
\lim_{s\to t}\int_X \dot{u}_{t}^+ (\omega +dd^c u_{s})^n   = \int_X \dot{u}_{t}^+ (\omega +dd^c u_{t})^n. 
\]
Combining the above we get
$$E(u_1)-E(u_0)  \geq  \int_X \dot{u}_{t}^+ (\omega +dd^c u_{t})^n.$$

\noindent For the reverse inequality we first use the concavity of $E$ to infer that for $h>0$ small we have
\[
E(u_1)-E(u_0) =\frac{E(u_{t+h})- E(u_t)}{h} \leq \int_X \frac{u_{t+h}-u_t}{h} (\omega+dd^c u_t)^n.  
\]
Letting $h\to 0^+$ we obtain
\[
E(u_1)-E(u_0) \leq \int_X \dot{u}_t^+ (\omega+dd^c u_t)^n, \quad \forall t\in (0,1).
\]
In a similar way we prove 
$$E(u_1)-E(u_0) = \int_X \dot{u}_t^- (\omega+dd^c u_t)^n, \quad \forall t\in (0,1). $$
We thus get the equality \eqref{eq: slope MA energy 1} for all $t\in (0,1)$. In particular we see that the left and right derivatives of $u_t$ are equal for all $t\in (0,1)$ and almost everywhere on $X$ with respect to $\omega_{u_t}^n$. 
 \end{proof}
 
 \begin{theorem}\label{thm: geodesic distance}
Fix $p\geq 1$ and $u_0,u_1\in \Ent(X,\omega)$. Assume $u_0-u_1$ is bounded. Then
\begin{equation}\label{eq: geodesic distance dp}
 \int_X |\dot{u}_t^+|^p \omega_{u_t}^n=\int_X |\dot{u}_t^-|^p \omega_{u_t}^n\;	\text{is constant in}\; t\in [0,1].
\end{equation}
If in addition $u_0,u_1\in \cE^p(X,\omega)$ then, for all $t\in [0,1]$,
\begin{equation}\label{eq: geodesic distance dp 1}
d_p^p(u_0,u_1) = \int_X |\dot{u}_t^+|^p \omega_{u_t}^n=\int_X |\dot{u}_t^-|^p \omega_{u_t}^n.
\end{equation}
\end{theorem}
The result was known when $u_0,u_1$ have bounded Laplacian (see \cite{Chen00},  \cite{Dar17AJM}).  It was pointed out by Darvas \cite{Dar17AJM} that the above result does not hold in general (even for bounded $u_0,u_1$). 

\begin{proof}
As we already noticed the convexity of $t\mapsto \cM_{\omega}(u_t)$ on $[0,1]$ implies that $\Ent(\omega^n, \omega_{u_t}^n)$ is bounded. In particular, for $t\in [0,1]$, we have $(\omega+dd^c u_t)^n=f_t \omega^n$ with $f_t\in L^1(X,\omega^n)$.  

Now, for $t\in \{0,1\}$, let $(u_{t,j})_j$ be sequences of smooth strictly $\omega$-psh functions such that  $u_{t,j} \searrow u_t$ as $j\to +\infty$, and 
\[
\Ent(\omega^n,(\omega+dd^c u_{t,j})^n)  \leq C\qquad t=0,1 ,
\]
for a constant $C>0$. The existence of these sequences were proved in  \cite{DL20GT} using the regularizing property of the Monge-Amp\`ere flow \cite{GZ17}, \cite{DnL17AIM}. As consequence of the construction, since $u_1-C\leq u_0\leq u_1+C$, the maximum principle ensures that $|u_{0,j}- u_{1,j}|\leq C$, for a uniform constant $C>0$. It follows from \cite{BBEGZ19} that $E(u_{t,j})$ is uniformly bounded. Using \cite[Lemma 1.7 and Proposition 1.8]{BBGZ13} and the bound $-C\omega \leq {\rm Ric}(\omega)\leq C\omega$, we then see that $E_{{\rm Ric(\omega)}}(u_{t,j})$ is also uniformly bounded.  It thus follows  that $\cM(u_{0,j})$ and $\cM(u_{1,j})$ are uniformly bounded in $j$. 

  For each $j$, let $(u_{t,j})_{t\in [0,1]}$ be the unique  geodesic joining $u_{0,j}$ to $u_{1,j}$. The convexity of the K-energy $\cM_{\omega}$ ensures that $\cM_{\omega}(u_{t,j})$ is uniformly bounded. It thus follows that the entropy is also uniformly bounded:
 $$\Ent(\omega^n, (\omega+ dd^c u_{t,j})^n) \leq C \qquad \forall t\in[0,1].$$  
 Since, for each $t\in[0,1]$ fixed, $u_{t,j} \searrow u_t$ it follows that the Monge-Amp\`ere measures converge weakly: $(\omega+dd^c u_{t,j})^n \rightharpoonup (\omega+dd^c u_t)^n$. The lower-semicontinuity of the entropy (with respect to the weak convergence) reads as
 $$\liminf_{j\rightarrow +\infty} \Ent(\omega^n,\omega_{u_{t,j}}^n) \geq \Ent(\omega^n,\omega_{u_t}^n ).$$
 Thus $\Ent(\omega^n,\omega_{u_t}^n)$ is also uniformly bounded, which implies in particular that $\omega_{u_t}^n$ has $L^1$ density. 
 
 By \cite{Dar17AJM}, 
 \[
d_p^p(u_{0,j},u_{1,j})= \int_X |\dot{u}_{t,j}|^p (\omega+ dd^c u_{t,j})^n 
 \]
 does not depend on $t\in [0,1]$, and $d_p^p(u_{0,j},u_{1,j}) \to d_p^p(u_0,u_1)$ if $u_0,u_1\in \cE^p(X,\omega)$. To prove \eqref{eq: geodesic distance dp} and \eqref{eq: geodesic distance dp 1} for $t\in (0,1)$ it then suffices to prove the following claim:
 \[
\int_X |\dot{u}_{t,j}|^p (\omega+ dd^c u_{t,j})^n \to \int_X |\dot{u}_t^{+}|^p \omega_{u_t}^n = \int_X |\dot{u}_t^{-}|^p \omega_{u_t}^n. 
\]
Our proof of the claim is similar to that of  \cite[Lemma 10.2]{Lem21}. 
Let $U$ be the set of all  $x\in X$ with $P_{\omega}(u_0,u_1)(x) >-\infty$. Fix $h>0$ and $x\in U$, by convexity of $t\mapsto u_{t,j}(x)$ we have  
\[
a_{t,h,j}(x):= \frac{u_{t,j}(x)-u_{t-h,j}(x)}{h} \leq \dot{u}_{t,j}(x) \leq \frac{u_{t+h,j}(x)-u_{t,j}(x)}{h} =: b_{t,h,j}(x).  
\]
Setting $c_{t,h,j} := \dfrac{a_{t,h,j}+b_{t,h,j}}{2}$ we then have 
\[
|\dot{u}_{t,j} - c_{t,h,j}| \leq \frac{|b_{t,h,j}-a_{t,h,j}|}{2}= \frac{| u_{t+h,j}-  u_{t-h,j}-2u_{t,j}|}{2h},
\]
giving 
 \[
\int_X |\dot{u}_{t,j}-c_{t,h,j}|^p (\omega+ dd^c u_{t,j})^n \leq \int_X \frac{|u_{t+h,j}+u_{t-h,j}-2u_{t,j}|^p}{(2h)^p} (\omega+dd^c u_{t,j})^n.
\]
Note that $f_j:=|u_{t+h,j}+u_{t-h,j}-2u_{t,j}|^p$ is quasi-continuous and uniformly bounded (since $u_{0,j}-u_{1,j}$ is uniformly bounded) and it does converge in capacity to $f:= |u_{t+h}+u_{t-h}-2u_{t}|^p$, that is quasi-continuous and bounded as well.\\
Using Lemma \ref{lem: conv measure} and Lemma \ref{lem: entropy vol} (which can be applied since $\Ent(\omega^n, \omega_{u_{t,j}}^n)$ and $\Ent(\omega^n,\omega_{u_{t}}^n)$ are uniformly bounded) we obtain 
 \[
\limsup_{j\to +\infty}\int_X |\dot{u}_{t,j}-c_{t,h,j}|^p (\omega+ dd^c u_{t,j})^n \leq \int_X \frac{|u_{t+h}+u_{t-h}-2u_t|^p}{(2h)^p} (\omega+dd^c u_t)^n.
\]
Letting $h\to 0^+$ we then have, for $t\in (0,1)$, 
\[
\lim_{h\to 0^+}\limsup_{j\to +\infty}\int_X |\dot{u}_{t,j}-c_{t,h,j}|^p (\omega+ dd^c u_{t,j})^n =\int_X |\dot{u}_t^+- \dot{u}_t^-|^p (\omega+dd^c u_t)^n =0,
\]
where the last equality is a consequence of Lemma \ref{lem: geodesic distance}.

Moreover, again by Lemma \ref{lem: conv measure} and Lemma \ref{lem: entropy vol}  we also have that for $t\in (0,1)$, 
\[
\lim_{j\to +\infty}\int_X |c_{t,j,h}|^p (\omega+dd^c u_{t,j})^n  = \int_X \frac{|u_{t+h}-u_{t-h}|^p}{(2h)^p} (\omega+dd^c u_{t})^n. 
\]

Letting $h\to 0^+$ we then have, for $t\in (0,1)$, 
\[
\lim_{h\to 0^+} \lim_{j\to +\infty}\int_X |c_{t,j,h}|^p (\omega+dd^c u_{t,j})^n  = \int_X \left| \frac{\dot{u}_{t}^+ +\dot{u}_{t}^-}{2}\right|^p (\omega+dd^c u_{t})^n. 
\]
Combining all these estimates we get

 \[
\int_X |\dot{u}_{t,j}|^p (\omega+ dd^c u_{t,j})^n \to \int_X 
\left| \frac{\dot{u}_{t}^+ +\dot{u}_{t}^-}{2}\right|^p \omega_{u_t}^n =\int_X 
\left|\dot{u}_{t}^+ \right|^p \omega_{u_t}^n = \int_X 
\left|\dot{u}_{t}^- \right|^p \omega_{u_t}^n,
\]
where the last equalities follows again from Lemma \ref{lem: geodesic distance}.
Since for each $j$ the left-hand side above is independent of $t$, it follows from the above that the right-hand side is also independent of $t$. We then conclude that \eqref{eq: geodesic distance dp} holds for all $t\in (0,1)$.

To prove \eqref{eq: geodesic distance dp} and \eqref{eq: geodesic distance dp 1} for $t=0$ we  proceed as follows. Observe that by \eqref{eq: Lip} $\dot{u}_t$ is bounded.  Fixing $h>0$ small enough, we have
\[
\int_X ||\dot{u}_t^{+}|^p -|\dot{u}_0^{+}|^p| f_t \omega^n  \leq C_p \int_X |\dot{u}_t^{+}-\dot{u}_0^{+}| f_t \omega^n \leq C_p \int_X \left(\frac{u_{t+h}-u_t}{h} -\dot{u}_0^{+} \right) f_t \omega^n, 
\]
where we used that for $a,b\in \mathbb{R}$, there exists a constant $C_p>0$ such that $||a|^p- |b|^p|\leq C_p|a-b|$. Observe that last inequality makes sense since $\dot{u}_t^{\pm}$ are bounded.

Since the entropy of $f_t$ is uniformly bounded and the functions in the integral are uniformly bounded, by Lemma \ref{lem: entropy vol} and Lemma \ref{lem: convergence bounded} we have 
\[
\limsup_{t\to 0^+}\int_X ||\dot{u}_t^{+}|^p -|\dot{u}_0^{+}|^p| f_t \omega^n \leq C_p \int_X \left(\frac{u_{h}-u_0}{h} -\dot{u}_0^+ \right) f_0 \omega^n. 
\]
Now, letting $h\to 0^+$ we obtain 
\[
\lim_{t\to 0^+}\int_X ||\dot{u}_t^{+}|^p -|\dot{u}_0^{+}|^p| f_t \omega^n =0. 
\]
By Lemma \ref{lem: entropy vol} and Lemma \ref{lem: convergence bounded} we also have 
\[
\lim_{t\to 0^+}\int_X |\dot{u}_0^{+}|^p (f_t-f_0) \omega^n =0. 
\]
It thus follows that 
\[
\lim_{t\to 0^+}\int_X |\dot{u}_t^{+}|^p(\omega+dd^c u_t)^n = \int_X |\dot{u}_0^{+}|^p (\omega+dd^c u_0)^n. 
\]
This proves \eqref{eq: geodesic distance dp} and \eqref{eq: geodesic distance dp 1} for $t=0$  and the same argument proves \eqref{eq: geodesic distance dp} and \eqref{eq: geodesic distance dp 1} for $t=1$, finishing the proof. 

\end{proof}

We thank L\'aszl\'o Lempert for his question which suggests the following consequence of  Theorem \ref{thm: geodesic distance}.  
\begin{corollary}\label{coro: dp distance entropy at t}
Fix $p\geq 1$ and $[0,1]\ni t\mapsto u_t\in \cE^p(X,\omega)$ a plurisubharmonic geodesic. Assume $u_0-u_1$ is bounded and $u_s\in \Ent(X,\omega)$ for some $s\in [0,1]$. Then \eqref{eq: geodesic distance dp 1} holds for that $s$. 
\end{corollary}
\begin{proof}
    {\bf Step 1:} we assume that $s=0$. By the same arguments we can also obtain the case $s=1$. 
    
    As in the proof of Theorem \ref{thm: geodesic distance}, let $v_{1,j}$ be a sequence of smooth $\omega$-psh functions decreasing to $u_1$, and set 
    \[
    u_{1,j} := P_{\omega}(u_0+C,v_{1,j}),
    \]
    where $C>0$ is such that $|u_0-u_1|\leq C$. Then $u_{1,j}\searrow P_\omega(u_0+C, u_1)=u_1$. Also, $\Ent(\omega^n, (\omega+dd^c v_{1,j})^n)<+\infty$ since $v_{1,j}$ is smooth. It then follows from Lemma \ref{lem: max principle} that $u_{1,j}\in \Ent(X,\omega)$ for all $j$ (even if the bound on the entropy of $(\omega+dd^c u_{1,j})^n$ is not uniform). Let $t\mapsto u_{t,j}$ be the psh geodesic connecting $u_0$ to $u_{1,j}$. Also, $u_0, u_1\in\mathcal{E}^p(X, \omega) $ by assumption and $u_{1,j}\in \mathcal{E}^p(X, \omega)$ as well since $u_{1,j}\geq u_1$. By Theorem \ref{thm: geodesic distance} we have 
    \[
    d_p^p(u_0,u_{1,j}) = \int_X |\dot{u}_{0,j}|^p (\omega+dd^c u_0)^n. 
    \]
    Since the sequence $u_{t,j}$ is decreasing in $j$ to $u_t$ and $u_{0,j} =u_0$ for all $j$, we have by convexity of $t \mapsto u_{t,j}(x)$ for fixed $x$ that 
    \[
    \dot{u}_0 \leq \dot{u}_{0,j} \leq \frac{u_{t,j}-u_0}{t}, \; t\in (0,1]. 
    \]
    Letting $j\to +\infty$ and then $t\to 0^+$ we obtain the convergence $\dot{u}_{0,j} \to \dot{u}_0$. Since $d_p(u_0,u_{1,j}) \to d_p(u_0,u_1)$ (see \cite{Dar15,Dar17AJM}) we obtain \eqref{eq: geodesic distance dp 1} for $s=0$. 
    
    \medskip
    
    {\bf Step 2:} Now for $s\in (0,1)$, we consider the geodesic $[0,1]\ni t \mapsto v_t$ connecting $u_0$ to $u_s$ and the geodesic $[0,1]\ni t \mapsto w_t$ connecting $u_s$ to $u_1$. By uniqueness of psh geodesics we have $v_t = u_{st}$ and $w_t= u_{(1-s)t + s}$.   By the first step we have 
    \[
    d_p^p(u_0,u_s)= \int_X |\dot{v}_1|^p (\omega+dd^c u_s)^n \quad {\rm and}\quad d_p^p(u_s,u_1)= \int_X |\dot{w}_0|^p (\omega+dd^c u_s)^n.
    \]
    Observe also that 
    \[
    \dot{v}_1= s\dot{u}_s^- \quad {\rm and}\quad   \dot{w}_0 =(1-s)\dot{u}_s^+,
    \]
    and that (since $u_t$ is a $d_p$-metric geodesic)
    \[
     d_p(u_0,u_s) = sd_p(u_0,u_1) \quad {\rm and}\quad  d_p(u_s,u_1) =(1-s) d_p(u_0,u_1).
    \]
    Combining all the above it follows that $$d_p(u_0,u_1)^p =  \int_X |\dot{u}_s^-|^p (\omega+dd^c u_s)^n = \int_X |\dot{u}_s^+|^p (\omega+dd^c u_s)^n.$$
    
\end{proof}

 \section{Big and nef classes}

 
 We fix a smooth closed $(1,1)$-form $\theta$ representing a big and nef cohomology class, and we denote by $\Omega$ the ample locus of $\{\theta\}$. Up to scaling we can assume that $\theta\leq \omega$.  The Darvas $d_p$ metrics have been extended to big and nef cohomology classes in \cite{DNL20} where it is proved that $(\cE^p(X,\theta), d_p)$ is a complete geodesic metric space and 
 \[
 d_p^p(u_0,u_1)= \int_X |\dot{u}_t|^p (\theta+dd^c u_t)^n, \quad  t\in \{0,1\}
 \]
 if $u_t= P_{\theta}(f_t)$, $t=0,1$, with $f_t$ smooth. The goal of this section is to prove that this formula holds for geodesic segments with less regular endpoints and for all $t\in[0,1]$. 
 
 We use the same ideas as in the K\"ahler case and for this reason we introduce a K-energy functional for a big and nef class and we need to prove that this is convex in $t$.  
 
 We start with a lemma:
 
 \begin{lemma}\label{dp_comp}
 Assume $u,v,w\in \cE^p(X, \theta)$ and  $u\leq v\leq w$. Then $d_p(u,v)\leq d_p(u,w) $. 
 \end{lemma}
 \begin{proof}
 Assume first that $u=P_\theta(f)$ for some smooth function $f$. By \cite[Lemma 3.13]{DNL20} we know that
 $$ d_p^p(u, v) =\int_X |\dot{u}_0|^p \theta_u^n \quad {\rm and }\quad  d_p^p(u, w) =\int_X |\dot{w}_0|^p \theta_u^n, $$
 where $u_t$ and $w_t$ are the geodesics joining $u, v$ and $u,w$, respectively. Since $u_0=w_0\leq u_1\leq w_1$, we have $0\leq \dot{u}_0 \leq \dot{w}_0 $. This implies that $ d_p(u, v) \leq d_p(u,w)$. 
 
 We now prove the general case $u\in \cE^p(X, \theta)$. Let $f_j$ be a sequence of smooth functions decreasing to $u$ and set $u_j:=P_\theta (f_j) \searrow u$, $v_j=\max(u_j, v)\searrow v$, $w_j=\max(u_j, w)\searrow w$. By construction we have that $u_j\leq v_j\leq w_j$. It follows from the above that
 $$d_p(u_j, v_j) \leq d_p(u_j, w_j).$$ The conclusion then follows from \cite[Proposition 3.12]{DNL20}.
 \end{proof}
 \subsection{Convexity of the K-energy}
 We let $\phi \in \PSH(X,\theta)$ with minimal singularities be the unique solution of
 \[
 (\theta+dd^c \phi)^n = \vol(\theta)\,\omega^n, \; \sup_X \phi=0. 
 \]
Both existence and uniqueness are guaranteed by \cite[Theorem A and Theorem 4.1]{BEGZ10}.
We keep in mind that, when it is possible to define $\Ric(\theta_\phi)$, then by construction $\Ric(\theta_\phi)=\Ric(\omega)$.

 For $u\in \PSH(X,\theta)$ with minimal singularities we set  
 \begin{equation}
 	\label{eq: K-energy}
 	\cM_{\theta}(u) := \bar{S}_{\theta} E(\theta; u,\phi) - n E_{\Ric (\omega)} (\theta; u,\phi) + \Ent(\theta_{\phi}^n, \theta_u^n),
 \end{equation}
 where 
 \[
\bar{S}_{\theta}:= \frac{n}{\vol(\theta)}\int_{\Omega}\Ric(\omega) \wedge \theta_{V_{\theta}}^{n-1},
 \]
 \[
 E(\theta; u,\phi) := \frac{1}{(n+1)\vol(\theta)} \sum_{k=0}^n \int_{\Omega} (u-\phi) \theta_u^{k} \wedge \theta_{\phi}^{n-k},
 \]
 and 
 \[
  E_{\Ric(\omega)}(\theta; u,\phi) :=  \frac{1}{n\vol(\theta)} \sum_{k=0}^{n-1} \int_{\Omega} (u-\phi) \theta_u^{k} \wedge \theta_{\phi}^{n-k-1}\wedge \Ric(\omega). 
 \]
 The Monge-Amp\`ere products above are well defined in $\Omega$ because the potentials involved are locally bounded there. It is not clear how to define $E_{\Ric(\omega)}(\theta; u,\phi)$ for a general $u\in \cE^1(X,\theta)$ since the later condition does not imply  $u\in \cE^1(X,\omega)$ even if $\theta\geq 0$ (see \cite{Dn16}).  When $\theta=\omega$, we have $\phi=0$  and $\cM_{\theta}$ is the Mabuchi K-energy functional \cite{Mab87}; see Section \ref{K-en}. 
    \begin{theorem}\label{thm: convexity of Mabuchi nef}
    	Assume $u_0,u_1\in \Ent(X,\theta)$ have minimal singularities. Let $(u_t)_{t\in [0,1]}$ be the psh geodesic connecting $u_0$ and $u_1$. Then $t\mapsto \cM_{\theta}(u_t)$ is convex and continuous on $[0,1]$. 
    \end{theorem}

    \begin{proof}
     We write 
    \[
    (\theta+dd^c u_t)^n= f_t \omega^n\; \text{for}\quad  t=0,1,
    \]
     and we first assume $f_0,f_1$ are bounded. In the following we want to approximate the endpoints $u_0, u_1$ by a decreasing sequence $\{u_{t,j}\}$, $t=0,1$, of bounded $(\theta+2^{-j}\omega)$-psh functions such that the entropy and the energy of the approximants are converging.
     
     We approximate  $\{\theta\}$ by K\"ahler classes $\{\theta+2^{-j} \omega\}$.  	 For each $j>0$,  we let $\phi_{j}$ be the unique smooth $(\theta+2^{-j}\omega)$-psh function such that 
    \begin{equation}\label{presc_ricci}
    (\theta+ 2^{-j} \omega +dd^c \phi_{j})^n =\vol(\theta+2^{-j}\omega)\omega^n,  \; \sup_X \phi_{j}=0. 
    \end{equation}
    Recall that we have normalized $\omega$ by $\vol(\omega)=1$.  By \cite{Yau78}, $\phi_{j}$ is smooth and 
    \[
    \theta_{j}:= \theta+2^{-j} \omega+ dd^c \phi_{j}
    \]
    is a K\"ahler form. Moreover, it follows from \cite[Theorem 1.9]{DGG20} that $\phi_j-V_{\theta+2^{-j}\omega}$ is uniformly bounded.  By construction we then have $\Ric(\theta_{j})= \Ric(\omega)$.
    
    Using \eqref{mabuchiNotK} we let $\cM_{j}$ be the Mabuchi K-energy defined for potentials in $\cE^1(X,\theta+2^{-j}\omega )$. More precisely, for any $u\in \cE^1(X,\theta+2^{-j}\omega )$
    
    \begin{eqnarray}
 	\label{eq: K-energy eqs 1}
 	\cM_{j}(u) &=&\bar{S}_{\theta_{j}} E(\theta+ 2^{-j}\omega; u,\phi_{j}) - n E_{\Ric(\omega)} (\theta+ 2^{-j} \omega; u,\phi_{j}) \\
 	&&\nonumber  + \Ent(\theta_{j}^n, (\theta+ 2^{-j} \omega+dd^c u)^n). 
 \end{eqnarray}
   where we recall that $$\bar{S}_{\theta_{j}} = \frac{n}{\vol(\theta_j)} c_1(X) \cdot \{\theta_j\}^{n-1},$$ 
 \begin{flalign*}
 	&E(\theta+2^{-j} \omega ; u,\phi_j) := \\
 	&\frac{1}{(n+1)\vol(\theta_j)} \sum_{k=0}^n \int_{\Omega} (u-\phi_j) (\theta+2^{-j} \omega +dd^c u)^{k} \wedge (\theta+2^{-j} \omega + dd^c \phi_j)^{n-k}
  \end{flalign*}	
 and 
 \begin{flalign*}
 	 &E_{\Ric(\omega)}(\theta+2^{-j} \omega ; u,\phi_j):= \\
 	 & \frac{1}{n\vol(\theta_j)} \sum_{k=0}^{n-1} \int_{\Omega} (u-\phi_j) (\theta+2^{-j} \omega  +dd^c u)^{k}  \wedge (\theta+2^{-j} \omega  +dd^c \phi_j)^{n-k-1}\wedge \Ric(\omega).
 \end{flalign*}

 Fix $\alpha>0$ so small that $\int_X e^{-2\alpha v} dV <+\infty$ for all $v\in \PSH(X,\theta)$. 
     For $t\in \{0,1\}$, let $u_{t,j}$ be the unique bounded $(\theta+2^{-j}\omega)$-psh function such that 
     \[
     (\theta+ 2^{-j} \omega +dd^c u_{t,j})^n = e^{\alpha (u_{t,j} - u_t)}f_t \omega^n. 
     \]
     Observe that the existence and uniqueness of a bounded solution $ u_{t,j}$ follows from \cite{Kol98} since $e^{-\alpha  u_t}f_t\in L^2(X, \omega^n)$ (recall that $f_0, f_1$ are bounded for the moment).\\
      It follows from the comparison principle  that $u_{t,j} \searrow u_t$, $t=0,1$ (see e.g. \cite[Lemma 2.5]{DDL1}). In particular we have $\sup_X u_{t,j}\leq \sup_X u_{t,1}$.  Since $f_0,f_1$ are bounded, it follows from \cite[Theorem 1.9]{DGG20} that $u_{t,j} -V_{\theta+2^{-j}\omega}$ is uniformly bounded. Hence $u_{t,j} -\phi_j$ is uniformly bounded, $t=0,1$.

 Let $(u_{t,j})_{t\in [0,1]} \subset \PSH(X,\theta+2^{-j} \omega)$ be the psh geodesic connecting $u_{0,j}$ to $u_{1,j}$ and $(u_t)_{t\in [0,1]} \subset \PSH(X,\theta)$ be the psh geodesic connecting $u_{0}$ to $u_{1}$. Observe that $u_{t,j}$ is decreasing to $u_t$ for any $t\in [0,1]$.
 Moreover, the fact that the psh geodesic is $t$-Lipschitz gives that for any $t\in[0,1]$
      $$|u_{t,j}-u_{0,j}|\leq t|u_{1,j}-u_{0,j}|\leq t|u_{1,j}-\phi_j|+t|u_{0,j}-\phi_j|\leq Ct.$$
      Since $|u_{0,j}-\phi_j|$ and $|u_{1,j}-\phi_j|$ are uniformly bounded, the above 
      means that $u_{t,j}-\phi_j$ is also uniformly bounded for all $t$.
      
 The advantage of the definition \eqref{eq: K-energy eqs 1} is that, by \cite{BB17,CLP16,BDL17}, the function $t\mapsto \cM_{j}(u_{t,j})$ is convex and continuous on $[0,1]$.  \\
 The convexity and continuity of $\cM_{\theta}$ follow if we can show that, for $t=0,1$, $\cM_{j}(u_{t,j})$ converges to $\cM_{\theta}(u_t)$, and for  $t\in (0,1)$ 
    \[
    \liminf_{j \to +\infty} \cM_{j}(u_{t,j}) \geq \cM_{\theta}(u_t).
    \] 
    We then need to study both the convergence of the energy part and the convergence of the energy as $j\rightarrow +\infty$.
    \smallskip

   {\it Convergence of the Monge-Amp\`ere energy.}  Fix $t\in [0,1]$ and $0\leq k \leq n$.  To simplify the notations we write 
    \[
    \mu_{t,j}:=(\theta+ 2^{-j} \omega +dd^c u_{t,j})^k \wedge \theta_{j}^{n-k},\; 
     \mu_t := (\theta +dd^c u_t)^k \wedge (\theta+dd^c \phi)^{n-k}.
     \]
   By Theorem \ref{thm: stability of solutions}, $\phi_{j} \to \phi$ in capacity as $j\to+\infty$. Also, we note that $u_t-\phi$ is bounded (for all $t$) since $u_t$ and $\phi$ are $\theta$-psh functions with minimal singularities while we already observed that $u_{t,j}-\phi_{j}$ is uniformly bounded for all $t$. In particular there exists  $C>0$ such that 
   $$-C \leq u_{t} - \phi\leq C\quad  {\rm and} \quad  -C \leq u_{t,j} - \phi_{j}\leq C \;\; \forall j.$$ 
   It then follows from \cite[Theorem 4.26]{GZbook} that $\mu_{t,j} \to \mu_t$ and $(u_{t,j}-\phi_{j})\mu_{t,j} \to (u_t-\phi) \mu_t$ in the weak sense of measures on $\Omega$ and 
    \[
    \liminf_{j\to +\infty}\int_{\Omega} (u_{t,j}-\phi_{j}+C) \mu_{t,j}  \geq  \int_{\Omega} (u_t-\phi+C) \mu_t,
    \]
    and 
     \[
    \limsup_{j\to +\infty}\int_{\Omega} (u_{t,j}-\phi_{j}-C) \mu_{t,j}  \leq  \int_{\Omega} (u_t-\phi-C) \mu_t. 
    \]
Since $\int_\Omega  d\mu_{t,j} = \vol(\theta+2^{-j} \omega)$, $\int_\Omega  d\mu_{t} = \vol(\theta)$ and $\vol(\theta+2^{-j} \omega) \to \vol(\theta) $, it follows from above that
$$ \liminf_{j\to +\infty}\int_{\Omega} (u_{t,j}-\phi_{j}) \mu_{t,j} = \int_{\Omega} (u_t-\phi) \mu_t.$$
Hence $$E(\theta+2^{-j}\omega; u_{t,j}, \phi_{j}) \to E(\theta; u_t,\phi)\qquad {\rm as}\; j\rightarrow +\infty$$
    The same arguments prove the convergence of $E_{\Ric(\omega)}(\theta+2^{-j}\omega; u_{t,j}, \phi_{j})$ towards $E_{\Ric(\omega)} (\theta; u_t,\phi)$. 
    
    \medskip

        {\it Convergence of the entropy term.} 
        We  write 
        \[
        \nu_{t,j}:= (\theta+2^{-j}\omega+dd^c u_{t,j})^n, \quad \nu_t := (\theta+dd^c u_t)^n.
        \]
         Observe that, since $u_{t,j}$ is decreasing to $u_t$, $\nu_{t,j}$ converges to $\nu_t$ in the weak sense of measures.
         
         Fix $t\in \{0,1\}$. Recall that $0\leq f_t$ is bounded and $\alpha>0$ is chosen such that $e^{-2\alpha u_t}\in L^1(X,\omega^n)$. A simple computation gives 
    \begin{eqnarray*}
    	 &&\Ent(\theta_j^n,\nu_{t,j}) = \frac{1}{\vol(\theta_{j})} \int_X \log \left( \frac{f_t e^{\alpha(u_{t,j}-u_t)}}{\vol(\theta_j)}\right) f_t e^{\alpha(u_{t,j}-u_t)}\omega^n\\
    	 &&= \frac{1}{\vol(\theta_{j})} \int_X \left (\log f_t -\log (\vol(\theta_j)) + \alpha(u_{t,j}-u_t)  \right) f_t e^{\alpha(u_{t,j}-u_t)}\omega^n.
    \end{eqnarray*}
    For $t=0,1$, since $f_t$ is bounded we have
    $$  f_t \log f_t \, e^{\alpha(u_{t,j}-u_t)}\leq C e^{-\alpha u_t} \in L^2(X, \omega^n) \subset L^{1}(X, \omega^n) $$ and
    $$\alpha(u_{t,j}-u_t)  f_t e^{\alpha(u_{t,j}-u_t)}\leq C (-\alpha u_t) e^{-\alpha u_t}\leq C e^{-2 u_t}\in  L^{1}(X, \omega^n). $$
    Also $-\log (\vol(\theta_j)) f_t e^{\alpha(u_{t,j}-u_t)}\leq 0$. By the dominated convergence theorem we can infer that $\Ent(\theta_j^n,\nu_{t,j})$ converges to $\Ent(\theta_{\phi}^n,\nu_t)$.

    We thus have the convergence $\cM_j(u_{t,j}) \to \cM_\theta (u_t)$, for $t=0,1$.
    Also, for $t\in (0,1)$, by lower semicontinuity of the entropy we have  
    \begin{eqnarray*}
    	 \liminf_{j\to+\infty} \Ent(\theta_j^n, \nu_{t,j}) 
   & \geq &   \Ent(\theta_{\phi}^n,\nu_t).
   \end{eqnarray*}
   
  This proves that $\liminf_{j\to+\infty} \cM_j(u_{t,j}) \geq \cM_\theta (u_t)$ for all $t$.
  This gives the convexity  of $[0,1] \ni t \mapsto \cM(u_t)$. 
  The above arguments also show that $t \mapsto \cM(u_t)$ is lower semicontinuous in $[0,1]$, hence it is continuous in $[0,1]$. 
  \smallskip

We now remove the boundedness assumption on $f_0,f_1$.  For each $t\in \{0,1\}$  we solve 
\[
(\theta +dd^c u_{t,j})^n =e^{\alpha(u_{t,j} -u_t)} \min(f_t,j) \omega^n.
\]
Then $C\geq u_{t,j} \searrow u_t$. Thus $\theta_{u_{t,j}}^n\to \theta_{u_t}^n$ in the weak sense of measures. Also $C\geq u_{t,j} \searrow u_t $ together with the fact that $u_t$ has minimal singularities imply that $u_{t,j}$ has minimal singularities as well and that we have a uniform bound $0\leq u_{t,j}-u_t\leq C$
Then we infer 
\begin{eqnarray*}
\Ent(\theta_{\phi}^n,\theta_{u_{t,j}}^n)&=&  \frac{1}{\vol(\theta)} \int_X \log \left( \frac{\min(f_t,j) e^{\alpha(u_{t,j}-u_t)}}{\vol(\theta_j)}\right) \min(f_t,j) e^{\alpha(u_{t,j}-u_t)}\omega^n\\
&\leq & \frac{C}{\vol(\theta)} \int_X f_t \log f_t \omega^n,
\end{eqnarray*}
where the last integral is bounded for $t=0,1$ since $\theta_{u_0}^n=f_0\,\omega^n, \theta_{u_1}^n=f_1\,\omega^n$ have finite entropy.
Again the dominated convergence theorem ensures that 
\begin{equation}\label{convEntmin}
\Ent(\theta_{\phi}^n,\theta_{u_{t,j}}^n) \to \Ent(\theta_{\phi}^n,\theta_{u_t}^n)\quad t=0,1.
\end{equation}
Let $(u_{t,j})_{t\in [0,1]}$ be the psh geodesic joining $u_{0,j}$ to $u_{1,j}$.  By the first step $t\mapsto \cM_{\theta}(u_{t,j})$ is convex and continuous on $[0,1]$. The arguments in the first steps can be applied to prove the  convergence of the Monge-Amp\`ere energies $E(\theta; u_{t,j},\phi)$, $E_{\Ric(\omega)}(\theta;u_{t,j},\phi)$ for $t\in [0,1]$. Also, \eqref{convEntmin} together with the lower semicontinuity of the entropy implies that $$	 \liminf_{j\to+\infty} \Ent(\theta_{\phi}^n,\theta_{u_{t,j}}^n) \geq   \Ent(\theta_{\phi}^n,\theta_{u_t}) \quad \forall t\in [0,1].$$
We thus get the convexity and continuity of $[0,1] \ni t \mapsto \cM_{\theta}(u_t)$.
    \end{proof}
    
    \subsection{Bounding the entropy along geodesics}
    \begin{prop}
    	\label{prop: entropy along geodesics big nef}
    	Assume $u_0,u_1\in \PSH(X,\theta)$ have the same singularities, i.e. that there exists $C>0$ such that $-C+u_1 \leq u_0\leq u_1+C$ on $X$. Let $(u_t)_{t\in [0,1]}$ be the psh geodesic joining $u_0$ to $u_1$.  Assume  $\Ent(\omega^n,\theta_{u_t}^n) \leq C$, for $t\in \{0,1\}$. Then, for some positive constant $C'$ depending on $C$, $X,\theta,\omega,n$ we have
    	\[
    	\Ent(\omega^n, (\theta +dd^cu_t)^n) \leq C', \ \forall t\in [0,1]. 
    	\] 
    \end{prop}
    
    \begin{proof}
        	   We write $(\theta+dd^c u_t)^n = f_t \omega^n$, for $t=0,1$ and solve, for each $j>0$
    	\[
    	(\theta+dd^c u_{t,j})^n= c(t,j) \min(f_t,j) \omega^n,\; u_{t,j}\in \cE(X,\theta),
    	\]  
    	with $\sup_X u_{t,j}=\sup_X u_t$.
    	By \cite[Theorem A and Theorem 4.1]{BEGZ10}, $u_{t,j}$ exists and have minimal singularities.
    	Here $c(t,j)$ is a normalization constant satisfying 
    	$$c(t,j)\int_X \min(f_t,j) \omega^n = \vol(\theta).$$ 
    	Since $\min(f_t,j) \nearrow f_t$, we have $c(t,j) \searrow 1$ as $j\to +\infty$. After extracting a subsequence we can assume that $u_{t,j}$ converges to $v_t$ in $L^{1}(X, \omega^n)$. By Corollary \ref{cor: BBEGZ entropy}, we have $d_1(u_{t,j},v_t)\to 0$ as $j\to +\infty$, for $t=0,1$. In particular $\theta_{u_{t,j}}^n \to \theta_{v_t}^n$ and so $(\theta+dd^c v_t)^n = f_t \omega^n$. By uniqueness $v_t=u_t+C_t$. But by construction $\sup_X v_t=\sup_X u_t$, hence $C_t=0$. It then follows that $d_1(u_{t,j},u_t)\to 0$ as $j\to +\infty$.
    
    	We also have 
    	\[
    	\Ent(\omega^n, (\theta+dd^c u_{t,j}))^n \leq C_1, \; t=0,1, \; j >0,
    	\]  	
    	for a uniform constant $C_1$. We next consider 
    	\[
    	\varphi_{0,j}:= P_{\theta}(u_{0,j}, u_{1,j}+C), \; \varphi_{1,j}:= P_{\theta}(u_{0,j}+C, u_{1,j}),
    	\] 
    	and observe that $\varphi_{0,j} -C\leq \varphi_{1,j} \leq \varphi_{0,j}+C$. By Lemma \ref{lem: MA contact} we have
    	\[
    	(\theta+dd^c \varphi_{t,j})^n \leq (c(0,j)f_0 + c(1,j) f_1) \omega^n, \quad  t=0,1. 
    	\]
	We thus have $\Ent(\omega^n, (\theta+dd^c \varphi_{t,j})^n)\leq C_2$, for a uniform constant $C_2$. Observe also that  
	\[
	P_{\theta}(u_0,u_1+C)=u_0\; \text{and} \; P_{\theta}(u_0+C,u_1)=u_1.
	\] 
	It then follows from \cite[Corollary 3.5]{DDL3} that
	\begin{eqnarray*}
	d_1(\varphi_{0,j},u_0) &=& d_1(P_{\theta}(u_{0,j}, u_{1,j}+C), P_{\theta}(u_0,u_1+C)) \\
	&\leq &   d_1(P_{\theta}(u_{0,j}, u_{1,j}+C), P_{\theta}(u_0,u_{1,j}+C)) +d_1(P_{\theta}(u_0,u_{1,j}+C),P_{\theta}(u_0,u_1+C))  \\
	&\leq  & d_1(u_{0,j}, u_0) +  d_1(u_{1,j}, u_1) \to 0.
	\end{eqnarray*}
Similarly we get that $	d_1(\varphi_{1,j},u_1)\to 0$.
	    	  
	 Let $\f_{t,j}$ be the psh geodesic connecting $\varphi_{0,j}$ to $\varphi_{1,j}$. By convexity of $d_1$, we have $d_1(\f_{t,j},u_t) \to 0$ for all $t\in (0,1)$. In particular $\theta_{\f_{t,j}}^n\to \theta_{u_t}^n$. Fixing $0\leq s<t\leq 1$,  since $\varphi_{0,j} - \varphi_{1,j}$ is uniformly bounded, by \eqref{eq: Lip} we have that $\f_{t,j}-\f_{s,j}$ is also uniformly bounded.\\
	 It then follows that
	 \[
	 E_{\Ric(\omega)}(\theta;\f_{t,j},\phi)-  E_{\Ric(\omega)}(\theta;\f_{s,j},\phi) = \frac{1}{n\vol(\theta)} \sum_{k=0}^{n-1} \int_{\Omega} (\f_{t,j}-\f_{s,j}) \theta_{\f_{t,j}}^k \wedge \theta_{\f_{s,j}}^{n-k-1} \wedge \Ric(\omega)
	 \]
	 is uniformly bounded. Similarly,
	 \[
	 E(\theta;\f_{t,j},\phi)-  E(\theta;\f_{s,j},\phi) = \frac{1}{(n+1)\vol(\theta)} \sum_{k=0}^{n} \int_{\Omega} (\f_{t,j}-\f_{s,j}) \theta_{\f_{t,j}}^k \wedge \theta_{\f_{s,j}}^{n-k-1} 
	 \]
	 is also uniformly bounded.
	 From the convexity of $t\mapsto \cM_{\theta}(\f_{t,j})$, which follows from Theorem \ref{thm: convexity of Mabuchi nef}, we thus infer
    	 \begin{eqnarray*}
    	 &&\Ent(\theta_{\phi}^n,(\theta+dd^c \f_{t,j})^n)  + E_{\Ric(\omega)}(\f_{t,j}) + E(\f_{t,j})\leq  \\
    	 && (1-t) \Ent(\theta_{\phi}^n,(\theta+dd^c \f_{0,j})^n)+(1-t)E_{\Ric(\omega)}(\f_{0,j}) + (1-t)E(\f_{0,j}) \\
    	 && +  t \Ent(\theta_{\phi}^n, (\theta+dd^c \f_{1,j})^n)+ tE_{\Ric(\omega)}(\f_{1,j}) + tE(\f_{1,j}) 
    	 \end{eqnarray*}
     
Now, it follows from the above that 
    \[
    E_{\Ric(\omega)}(\f_{t,j}) - (1-t)E_{\Ric(\omega)}(\f_{0,j}) -tE_{\Ric(\omega)}(\f_{1,j}) 
    \]
    is uniformly bounded as well as the other energy part. Since $\theta_{\phi}^n= \vol(\theta)\omega^n$, we have
    	 \[
    	 \Ent(\omega^n, (\theta+dd^c \f_{t,j})^n) \leq (1-t) \Ent(\omega^n,(\theta+dd^c \f_{0,j})^n) +t \Ent(\omega^n,(\theta+dd^c \f_{1,j})^n) + C_3. 
    	 \]
    	 Now, the right hand side is uniformly bounded while $$\liminf_{j\to +\infty}\Ent(\omega^n, (\theta+dd^c \f_{t,j})^n)\geq \Ent(\omega^n, (\theta+dd^c u_{t})^n).$$ We then arrive at the conclusion. 
    \end{proof}
    
    \subsection{Geodesic distance}
   Using Proposition \ref{prop: entropy along geodesics big nef} we can repeat the arguments of Theorem \ref{thm: geodesic distance} to obtain the following extension: 
     \begin{theorem}\label{thm: geodesic distance big nef}
Assume $\theta$ is a smooth closed real $(1,1)$-form representing a big and nef class. Fix $p\geq 1$,  $u_0,u_1\in \Ent(X,\theta)$ and let $u_t$ be the psh geodesic connecting $u_0$ to $u_1$. If $u_0-u_1$ is bounded then 
\begin{equation}\label{eq: geodesic distance dp big nef}
 \int_X |\dot{u}_t|^p (\theta +dd^c u_t)^n\; \text{is constant in}\; t \in [0,1].
\end{equation}
If in addition $u_0,u_1\in \cE^p(X,\theta)$, then 
$$d_p^p(u_0,u_1)= \int_X |\dot{u}_t|^p (\theta +dd^c u_t)^n,\quad \forall t\in [0,1].$$ 
\end{theorem}
\begin{proof}
	By Proposition \ref{prop: entropy along geodesics big nef}, $\Ent(\omega^n,\theta_{u_t}^n)$ is uniformly bounded. The proof of Lemma \ref{lem: geodesic distance} can thus be repeated word by word giving that, for each $t\in (0,1)$, $\dot{u}_t^+=\dot{u}_t^-$ almost everywhere on $X$ with respect to $(\theta+dd^c u_t)^n$. 
	
	Fix $t\in \{0,1\}$. For each $j>0$ let $v_{t,j}\in \cE(X,\theta+2^{-j}\omega)$ be the unique solution to 
	\[
	(\theta+2^{-j}\omega +dd^c v_{t,j})^n = \vol(\theta+2^{-j}\omega) (\theta+dd^c u_t)^n, \; \sup_X v_{t,j}=\sup_X u_t.
	\]
	By Theorem \ref{thm: stability of solutions}, after passing to a subsequence,  $v_{t,j}$ can be sandwiched between two monotone sequences converging to $u_t$: 
	\[
	\cE(X,\theta) \ni h_{t,j}\leq v_{t,j}\leq w_{t,j}\in \cE(X,\theta+2^{-j}\omega).
	\] 
Fix a constant $C>0$ such that $|u_0-u_1|\leq C$.  Define 
	$$u_{0,j}:= P_{\theta+2^{-j}\omega}(v_{0,j},v_{1,j}+C), \qquad u_{1,j}:= P_{\theta+2^{-j}\omega}(v_{0,j}+C,v_{1,j}).$$
	By construction
	$$u_{1,j} -C\leq u_{0,j}\leq u_{1,j} -C$$
	and, as we observed in the proof of Proposition \ref{prop: entropy along geodesics big nef}, $\Ent(\omega^n, (\theta+2^{-j}\omega+dd^c u_{t,j})^n)$ is uniformly bounded.

	Let $(u_{t,j})_{t\in [0,1]}$ be the psh geodesic connecting $u_{0,j}$ to $u_{1,j}$. We claim that for all $t\in [0,1]$, $u_{t,j} \to u_t$ in capacity. Indeed, since $u_{0.j}$ and $u_{1,j}$ can be sandwiched between two monotone sequences converging to $u_0$ and $u_1$ (respectively), using the comparison principle \cite[Proposition 2.15]{DDL3} we can also sandwich $u_{t,j}$ for $t\in[0,1]$ between two monotone sequences converging to $u_t$. This proves the claim. 
	
	Thus, for each $t\in [0,1]$,  $(\theta+2^{-j}\omega+dd^c u_{t,j})^n \rightarrow (\theta+dd^c u_t)^n$ weakly.  By Theorem \ref{thm: geodesic distance}, 
	$$
	\int_X |\dot{u}_{t,j}|^p (\theta+2^{-j}\omega+dd^c u_{t,j})^n
	$$
	 is constant in $t\in [0,1]$. Letting $j\to +\infty$, and repeating the argument of Theorem \ref{thm: geodesic distance} we obtain that 
	$ \int_X |\dot{u}_{t}|^p (\theta+dd^c u_{t})^n$ is constant in $t\in [0,1]$.
	
	\medskip

	It remains to prove the last statement. Assume in addition that  $u_0, u_1\in \cE^p(X, \theta) $. We want to prove that 
	$$
	d_p^p(u_0, u_1)=  \int_X |\dot{u}_{t}|^p (\theta+dd^c u_{t})^n,\;  \forall t\in [0,1].
	$$ 
	We divide the proof into two steps.
	\medskip
	
\noindent 	\underline{Step 1}: We first assume $u_0, u_1$ have both minimal singularities. If $u_0=P_\theta(f)$ for some smooth function, it then follows from \cite[Lemma 3.13]{DNL20} that 
	$$ d_p^p(u_0, u_1)= \int_X |\dot{u}_0|^p \theta_{u_0}^n= \int_X |\dot{u}_t|^p \theta_{u_t}^n,$$
	where the last equality follows from the fact that the quantity at the right-hand-side is constant in $t$. 
	For the general case of two functions with minimal singularities, we consider $f_{0,j}$ a sequence of smooth functions decreasing to $u_0$ and $u_{0,j}=P_\theta(f_{0,j})$. By the above
	$$ d_p^p(u_{0,j}, u_1)= \int_X |\dot{u}_{0,j}|^p \theta_{u_{0,j}}^n = \int_X |\dot{u}_{1,j}|^p \theta_{u_1}^n,$$
	where $u_{t,j}$ is the geodesic joining $u_{0,j}$ and $u_1$. Let $u_t$ be the geodesic joining $u_0$ and $u_1$. Then $u_{t,j}\searrow u_t$ and 
	$$  
	\frac{u_1- u_{t,j}}{1-t} \leq  \dot{u}_{1,j}= \lim_{t\to 1^-}\frac{u_1- u_{t,j}}{1-t} \leq  \lim_{t\to 1^-} \frac{u_1-u_t}{1-t} =\dot{u}_1.
	$$
	Letting $j\to +\infty$, we get
	$$
	\frac{u_1- u_{t}}{1-t} \leq \liminf_{j\to +\infty} \dot{u}_{1,j} \leq \limsup_{j\to +\infty} \dot{u}_{1,j} \leq  \dot{u}_1.
	$$
	It follows that $\dot{u}_{1,j}$ converges pointwise to $\dot{u}_1$. 
	The dominated convergence theorem then ensures that $$ d_p^p(u_{0,j}, u_1) \rightarrow \int_X |\dot{u}_1|^p \theta_{u_1}^n= \int_X |\dot{u}_t|^p \theta_{u_t}^n, \; \forall t\in [0,1].$$
	On the other hand $d_p (u_{0,j}, u_1)\rightarrow d_p (u_{0}, u_1)$ by \cite[Proposition 3.12]{DNL20}. This completes the proof of Step 1. 
	\medskip
	
\noindent 	\underline{Step 2}: We now work in the general case $u_0, u_1\in \cE^p(X, \theta)\cap \Ent(X, \theta)$ and $|u_0-u_1|\leq C$. 

\medskip

Fix $t\in \{0,1\}$.  By assumption we have $(\theta+dd^c u_t)^n= f_t \omega^n$ with $\int_X f_t \log f_t \omega^n <+\infty$. We solve 
$$ (\theta+dd^c v_{t,j})^n= c_{t,j} \min(f_t,j) \omega^n, \quad \sup_X v_{t,j} = \sup_X u_t,
$$ 
where $c_{t,j}$ is a normalization constant. We have $c_{t,j} \searrow 1$ as $j\to +\infty$. By the proof of Theorem \ref{thm: stability of solutions}, after extracting a subsequence we can sandwich $v_{t,j}$ by two monotone sequences: 
\[
\cE^1(X,\theta)\ni \varphi_{t,j} \leq v_{t,j} \leq \psi_{t,j} \in \cE^1(X,\theta) \; , \; \varphi_{t,j} \nearrow u_{t}, \; \psi_{t,j} \searrow u_{t}. 
\]
Recall that $\varphi_{t,j}$ is defined by $\varphi_{t,j}= P_{\theta}(\inf_{k\geq j} v_{t,k})$ and obesrve that we can assume $c_{t,j}\leq 2$. By Lemma \ref{lem: MA contact} we have 
\[
(\theta+dd^c \f_{t,j})^n \leq  2 (\theta+dd^c u_{t})^n. 
\]
Since $u_t \in \cE^p(X,\theta)$, by Lemma \ref{MA_ep} we have that $\f_{t,j}\in \cE^p(X,\theta)$. 

We now define
$$
u_{0,j}:= P_{\theta}(v_{0,j}, v_{1,j}+C)\; , \; u_{1,j}:= P_\theta(v_{0,j}+C, v_{1,j}).
$$
By construction we have that $u_{0,j},u_{1,j}$ have minimal singularities, $|u_{0,j}- u_{1,j}|\leq C$, and  $\Ent (\omega^n, (\theta+dd^c u_{t,j})^n) \leq C'$ for $t\in \{0,1\}$. 

Moreover, we claim that $d_p(u_{t,j},u_t) \to 0$.  Indeed, define 
$$ \tilde{\f}_{0,j} = P_\theta (\f_{0,j}, \f_{1,j}+C ), \qquad \tilde{\f}_{1,j} = P_\theta ( \f_{1,j}+C, \f_{1,j} ),$$
$$ \tilde{\psi}_{0,j}=P_\theta ( \psi_{0,j}, \psi_{1,j}+C ),\qquad \tilde{\psi}_{1,j}=P_\theta ( \psi_{0,j}+C, \psi_{1,j}). 
	$$
For $t\in \{0,1\}$ we then have 
	\begin{eqnarray*}
	\cE^p(X, \theta) \ni \tilde{\f}_{t,j} \leq u_{t,j}\leq  \tilde{\psi}_{t,j}\in \cE^p(X, \theta),
	\end{eqnarray*}
and $ \tilde{\f}_{t,j}  \nearrow u_t$, $\tilde{\psi}_{t,j}\searrow u_t$. The fact that $\tilde{\f}_{t,j}$ and  $\tilde{\psi}_{t,j}$ belong to $\cE^p (X, \theta)$ follows from Lemma \ref{MA_ep} and Lemma \ref{lem: MA contact}. 
	By the triangle inequality and Lemma \ref{dp_comp}, we then have
	\begin{eqnarray*}
	d_p(u_{t,j}, u_t) & \leq & d_p(u_{t,j}, \tilde{\f}_{t,j}) + d_p (\tilde{\f}_{t,j}, u_t) \\
	 & \leq & d_p(\tilde{\varphi}_{t,j}, \tilde{\psi}_{t,j}) + d_p (\tilde{\f}_{t,j}, u_t)\\
	  & \leq & d_p(\tilde{\varphi}_{t,j}, u_t) +d_p(\tilde{\psi}_{t,j}, u_t) + d_p (\tilde{\f}_{t,j}, u_t)\to 0. 
	 \end{eqnarray*}
	  The claim is then proved. 
	  
	 Now, let $(u_{t,j})_{t\in [0,1]}$ be the psh geodesic joining $u_{0,j}$ to $u_{1,j}$. By the first step we can infer that, for $t\in [0,1]$,
	 $$ d_p^p(u_{0,j}, u_{1,j}) = \int_X |\dot{u}_{t,j}|^p \theta_{u_{t,j}}^n.$$
	 The left-hand side converges to $d_p^p(u_{0}, u_{1})$ thanks to the claim. It follows from Proposition \ref{prop: entropy along geodesics big nef} that $\Ent(\omega^n, (\theta+dd^c u_{t,j})^n)$ is uniformly bounded for $t\in [0,1]$, $j\geq 1$. Using the same arguments of Theorem \ref{thm: geodesic distance} we see that the right-hand side above converges to $\int_X |\dot{u}_{t}|^p \theta_{u_{t}}^n$. This finishes the proof. 
	\end{proof}

\section{Monge-Amp\`ere measures on contact sets}

As an application we prove the following result generalizing \cite{DNT19} in the case of a big and nef class.

\begin{theorem}\label{thm: MA contact}
Assume $\{\theta\}$ is big and nef, 
$u \in \PSH(X,\theta)$, $v\in \Ent(X,\theta)$ and $u\leq v$. 
Then 
 \[
 {\bf 1}_{\{u = v\}} \theta_u^n = {\bf 1}_{\{u=v\}} \theta_v^n. 
 \]
\end{theorem}

\begin{proof}
Without loss of generality we assume $v\leq 0$.

We start assuming that both $u$ and $v$ have finite entropy and $u-v$ is bounded. 
Let $u_t$ be the psh  geodesic connecting $u_0=u$ to $u_1=v$. Note that since $u_0\leq u_1$ we have that $u_0\leq u_t \leq u_1$ for all $t\in [0,1]$ and so 
$$
0\leq \dot{u}_t \leq \sup_X |u_1-u_0|<+\infty
$$ 
for all $t \in [0,1]$. By Theorem \ref{thm: geodesic distance big nef} we then have that 
\[
\int_X \dot{u}_0^p (\theta+dd^c u_0)^n = \int_X \dot{u}_1^p (\theta +dd^c u_1)^n, \ \forall p\geq 1. 
\]
 It thus follows that 
\[
\int_X e^{-A \dot{u}_0} (\theta +dd^c u_0)^n = \int_X e^{-A \dot{u}_1} (\theta +dd^c u_1)^n,  \ \forall A>0. 
\]
Letting $A\to +\infty$ we have that 
\begin{equation}
\label{id_middle}
\int_{\{\dot{u}_0 =0\}}  (\theta +dd^c u)^n = \int_{\{\dot{u}_1 =0\}}  (\theta +dd^c v)^n.
\end{equation}
  We do not necessarily have that $\{\dot{u}_0 =0\} = \{u_0=u_1\}$ but we do have 
  $$
  \{\dot{u}_1=0\} = \{u_0=u_1\}
  $$ 
  because $0\leq u_1-u_0 \leq \dot{u}_1$ (the last inequality follows from the convexity of $t\to u_t$)
  and if $u_0(x)=u_1(x)$ then $u_0(x)\leq u_t(x)\leq u_1(x)$ (hence equality) for all $t\in [0,1]$, implying that $\dot{u_1}(x)=0$.

 We now will prove the statement of the theorem in the case $u= P_\theta(f,v)$, where $f$ is a smooth negative function. Observe that in this case $u$ and $v$ belong to $\Ent(X,\theta)$. We also have that $u-v$ is bounded. Thus, by the first step, the equality \eqref{id_middle} holds.
 We also fix $\lambda\in (0,1)$ and  set $h_{\lambda}:= P(\lambda f,v)$. Observe that $\lambda u+ (1-\lambda)v$ is $\theta$-psh and it is smaller than $\min(\lambda f,v)$. We thus have 
\[
u \leq \lambda u+ (1-\lambda) v \leq h_{\lambda} \leq v.
\]
From this we infer $\{u=v\} = \{u=h_{\lambda}\}$. The inclusion $\subseteq$ is clear, while the other comes from the fact $u=h_\lambda$ implies $(1-\lambda)u=(1-\lambda) v$.

 Let $\psi_t$ be the psh geodesic connecting $\psi_0 = u$ to $\psi_1=h_{\lambda}$. Note that $\psi_0 \leq \psi_1$.  By assumption we have $\psi_0,\psi_1 \in \Ent(X,\theta)$ and $\psi_1-\psi_0$ is bounded. We can thus apply the first step to get
\begin{flalign}
	\int_{\{\dot{\psi}_0 =0\}}  \theta_u^n &= \int_{\{\dot{\psi}_1 =0\}}  (\theta +dd^c h_\lambda)^n = \int_{\{\psi_1 = u\}}  (\theta +dd^c  h_\lambda)^n \nonumber\\
	&=\int_{\{h_\lambda = u\}}  (\theta +dd^c  h_\lambda)^n
	= \int_{\{u=v\}}  (\theta +dd^c  h_\lambda)^n.\label{eq: MA contact 1}
\end{flalign}
Since $\psi_0=u_0$ and $\psi_1=h_\lambda\leq v=u_1$, by the comparison principle for geodesics (see e.g. \cite{DDL3})  we have $\psi_t \leq u_t$ for all $t$, hence $0\leq \dot{\psi}_0 \leq \dot{u}_0$. We thus have 
\begin{equation*}
	\{\dot{u}_0=0\}\subset \{\dot{\psi}_0 =0\}.
\end{equation*}
From this and \eqref{eq: MA contact 1} we obtain 
\[
\int_{\{u=v\}} (\theta +dd^c h_{\lambda})^n \geq \int_{\{\dot{u}_0 =0\}}  (\theta +dd^c u)^n = \int_{\{u=v\}} (\theta +dd^c v)^n,
\]
where the last equality is \eqref{id_middle}. 
By Lemma \ref{lem: max principle} and  $h_\lambda\leq v$ we have
$$ 
{\bf 1}_{\{h_\lambda=v\}} (\theta +dd^c h_{\lambda})^n  \leq {\bf 1}_{\{h_\lambda=v\}} (\theta +dd^c v)^n.
$$
The inclusion $\{u=v\}\subset \{h_{\lambda}=v\}$ then gives
\[
{\bf 1 }_{\{u=v\}} (\theta+dd^c h_{\lambda} )^n \leq  {\bf 1}_{\{u=v\}} (\theta +dd^c v)^n. 
\]
Since the  two measures have the same total mass, we have 
\[
{\bf 1 }_{\{u=v\}} (\theta+dd^c h_{\lambda})^n =  {\bf 1}_{\{u=v\}} (\theta +dd^c v)^n. 
\]
By construction we have $h_{\lambda} \searrow u$. 
We now let $\lambda \to 1^-$ and, since ($\{u=v\}$ is quasi-closed, Lemma \ref{lem: convergence quasi open} gives
\[
{\bf 1 }_{\{u=v\}} (\theta+dd^c u)^n \geq   {\bf 1}_{\{u=v\}} (\theta +dd^c v)^n.
\] The reverse inequality holds thanks to Lemma \ref{lem: max principle} since we assumed $u=P_\theta(f, v)\leq v$. 

To prove the general case we approximate $u$ from above by a decreasing sequence of smooth  functions $u_j\leq 0$ and argue as in \cite{DNT19}. 
Setting $\varphi_j := P_\theta(u_j,v)$ and using the first step we have 
\[
{\bf 1}_{\{\varphi_j=v\}} (\theta +dd^c \varphi_j)^n = {\bf 1}_{\{\varphi_j=v\}} (\theta +dd^c v)^n. 
\]
Observe that $u\leq \varphi_j \leq v$, hence $\{u=v\}\subset \{\varphi_j=v\}$. We then have, for all $A>0$, 
\[
\int_X e^{A(\varphi_j-v)} (\theta +dd^c \varphi_j)^n \geq \int_{\{u=v\}} (\theta +dd^c v)^n. 
\]
Since $\varphi_j \searrow u$ we know that $ (\theta +dd^c \varphi_j)^n$ converges weakly to $ (\theta +dd^c u)^n$. Also $(e^{A(\varphi_j-v)})_j$ is a uniformly bounded sequence of functions decreasing to $e^{A(u-v)}$ that is bounded as well. Letting $j\to +\infty$ we obtain 
\[
\int_X e^{A(u-v)} (\theta +dd^c u)^n \geq \int_{\{u=v\}} (\theta +dd^c v)^n. 
\]
Now, letting $A\to +\infty$ we obtain the result. 
\end{proof}

If two $\theta$-psh functions coincide almost everywhere with respect to Lebesgue measure then they coincide everywhere.  Using the above result we obtain the following: 
\begin{corollary}
Assume $\{\theta\}$ is big and nef, $v \in {\rm Ent}(X,\theta)$ and $u\in \PSH(X,\theta)$. If $u=v$ $\theta_v^n$-almost everywhere then $u=v$ everywhere. 
\end{corollary}

\begin{proof}
	By the domination principle \cite[Corollary 3.10]{DDL2} we have that $u\leq v$, hence $\int_X \theta_u^n \leq \int_X \theta_v^n$. Also, the previous result implies that ${\bf 1}_{\{u=v\}} \theta_v^n= {\bf 1}_{\{u=v\}}\theta_u^n$ and by assumption $\theta_v^n= {\bf 1}_{\{u=v\}} \theta_v^n$. Thus 
	\[
	\int_X \theta_v^n = \int_{\{u=v\}} \theta_u^n.
	\]
	Comparing the total mass we get $\theta_u^n =\theta_v^n$, hence by uniqueness $u=v+C$ form some $C\leq 0$ \cite{BEGZ10,DiwJFA09}. Now, $\int_X(u-v)\theta_v^n=0$ gives $C=0$.  
\end{proof}

\newcommand{\etalchar}[1]{$^{#1}$}
\providecommand{\bysame}{\leavevmode\hbox to3em{\hrulefill}\thinspace}
\providecommand{\MR}{\relax\ifhmode\unskip\space\fi MR }
\providecommand{\MRhref}[2]{%
  \href{http://www.ams.org/mathscinet-getitem?mr=#1}{#2}
}
\providecommand{\href}[2]{#2}

\end{document}